\documentclass[10pt, notitlepage]{article}
\usepackage{amsmath, amssymb, pgfplots,yfonts, setspace, enumitem,titlesec}
\usepackage[toc,page]{appendix}
\usepackage{hyperref}
\usepackage{pst-node}
\usepackage{tikz-cd} 
\usepackage[margin=1.5in]{geometry}
\usepackage{cite}

\def\tr{\mathop{\hbox{tr}}}

\def\id{\mathop{\hbox{id}}}

\def\Span{\mathop{\hbox{span}}}

\def\Pol{\mathrm{Pol}}

\def\fB{\mathcal{B}}
\def\fH{\mathcal{H}}

\def\fK{\mathcal{K}}

\def\fP{\mathcal{P}}

\def\C{\mathbb{C}}
\def\R{\mathbb{R}}
\def\H{\mathbb{H}}

\def\G{\mathbb{G}}
\def\K{\mathbb{K}}

\def\acts{\curvearrowright}
\usepackage{mathtools}

\makeatletter
\DeclareRobustCommand\widecheck[1]{{\mathpalette\@widecheck{#1}}}
\def\@widecheck#1#2{%
    \setbox\z@\hbox{\m@th$#1#2$}%
    \setbox\tw@\hbox{\m@th$#1%
       \widehat{%
          \vrule\@width\z@\@height\ht\z@
          \vrule\@height\z@\@width\wd\z@}$}%
    \dp\tw@-\ht\z@
    \@tempdima\ht\z@ \advance\@tempdima2\ht\tw@ \divide\@tempdima\thr@@
    \setbox\tw@\hbox{%
       \raise\@tempdima\hbox{\scalebox{1}[-1]{\lower\@tempdima\box
\tw@}}}%
    {\ooalign{\box\tw@ \cr \box\z@}}}
\makeatother
\usepackage[utf8]{inputenc}
\usepackage[english]{babel}
\usepackage{amsthm}

\theoremstyle{plain}
\newtheorem{thm}{Theorem}[section]
\newtheorem{lem}[thm]{Lemma}
\newtheorem{prop}[thm]{Proposition}
\newtheorem{cor}[thm]{Corollary}
\newtheorem{ques}[thm]{Question}

\theoremstyle{definition}
\newtheorem{defn}[thm]{Definition}
\newtheorem{rem}[thm]{Remark}
\newtheorem{eg}[thm]{Example}
\makeatletter
\newcommand*{\toccontents}{\@starttoc{toc}}
\makeatother
\titleformat{\chapter}
  {\normalfont\bfseries}{\thechapter}{1em}{}
\titlespacing*{\chapter}{0pt}{3.5ex plus 1ex minus .2ex}{2.3ex plus .2ex}
\pgfplotsset{compat=1.17}
\begin{document}
%Front Material
%%%%%%%%%%%%%%%%%%%%%%%%%%For the title page
%\input{thesis-front-pages}
\title{Tracial States and $\G$-Invariant States of Discrete Quantum Groups}

\author{Benjamin Anderson-Sackaney}
\maketitle
\begin{abstract}
    We investigate the tracial states and $\mathbb{G}$-invariant states on the reduced $C^*$-algebra $C_r(\widehat{\mathbb{G}})$ of a discrete quantum group $\mathbb{G}$. Here, we denote its dual compact quantum group by $\widehat{\mathbb{G}}$. Our main result is that a state on $C_r(\widehat{\G})$ is tracial if and only if it is $\G$-invariant. This generalizes a known fact for unimodular discrete quantum groups and builds upon the work of Kalantar, Kasprzak, Skalski, and Vergnioux. As one consequence of this, we find that $C_r(\widehat{\mathbb{G}})$ is nuclear and admits a tracial state if and only if $\mathbb{G}$ is amenable. This resolves an open problem due to C.-K. Ng and Viselter, and Crann, in the discrete case. As another consequence, we prove that tracial states on $C_r(\widehat{\G})$ ``concentrate'' on $\widehat{\G}_F$, where $\G_F$ is the cokernel of the Furstenberg boundary. Furthermore, given certain assumptions, we characterize the existence of traces on $C_r(\widehat{\mathbb{G}})$ in terms of whether or not $\widehat{\mathbb{G}}_F$ is Kac type. We also characterize the uniqueness of (idempotent) traces in terms of whether not $\widehat{\G}_F$ is equal to the canonical Kac quotient of $\widehat{\G}$. These results rely on  the following, of which we give proofs: So\l tan's canonical Kac quotient construction, whether it is applied to the universal or the reduced CQG $C^*$-algebra of $\widehat{\G}$ (when the latter admits a trace), yields the maximal Kac type closed quantum subgroup of $\widehat{\G}$.
\end{abstract}

\begin{center}
\bfseries{Contents}
\end{center}
\toccontents

\section{Introduction}
In \cite{BKKO17} and \cite{KK14}, group dynamical characterizations of the unique trace property and simplicity of reduced group $C^*$-algebras were achieved, where they showed that the reduced $C^*$-algebra $C_r(\widehat{G})$ of a discrete group $G$ has a unique tracial state if and only if  the action of $G$ on its Furstenburg boundary $\partial_F(G)$ is faithful. In turn, we also find that simplicity of $C_r(\widehat{G})$ implies it has a unique trace (the Haar state). Given that $R_a(G) = \ker(G \acts \partial_F(G))$, where $R_a(G)$ is the amenable radical, this shows that the unique trace property is equivalent to having $R_a(G) = \{e\}$. A more general result, which illuminates the above characterization of the unique trace property, was also obtained in \cite{BKKO17}, namely that every tracial state $\tau : C_r(\widehat{G})\to\C$ ``concentrates'' on $R_a(G)$, i.e., satisfies $\tau(\lambda_G(s)) = 0$ whenever $s\notin R_a(G)$. Essential to the proof is the well-known fact that a state on $C_r(\widehat{G})$ is tracial if and only if it is $G$-invariant, i.e., is invariant with respect to the conjugation action $G\acts C_r(\widehat{G})$.

Serving as a stepping stone towards establishing quantum group dynamic machinery for quantizing the unique trace property and simplicity of $C_r(\widehat{\G})$ where $\G$ is a discrete quantum group, Kalantar et al. \cite{KKSV22} constructed the Furstenberg boundary $\partial_F(\G)$ (see Section $4.2$ therein). Moreover, the cokernel of $\G\acts \partial_F(\G)$ was shown in \cite{KKSV22} to be equal to $\ell^\infty(\G_F)$ where $\widehat{\G}_F$ is a quantum subgroup of $\widehat{\G}$ that is minimal as an object where $\ell^\infty(\G_F)$ is relatively amenable (see Section $5.1$ therein). In particular, we might call $\G_F$ the `relatively amenable coradical' of $\G$. There is a catch, though. At the quantum level, the canonical trace is replaced with the Haar state $h_{\widehat{\G}}\in C_r(\widehat{\G})^*$, however, $h_{\widehat{\G}}$ may not be tracial and there are known examples where $C_r(\widehat{\G})$ has no trace (e.g., see \cite{B97}). It turns out the Haar state is tracial if and only if $\G$ is unimodular. Therefore, in \cite{KKSV22}, the unique trace property was considered for unimodular discrete quantum groups, where it was shown that if the action of $\G$ on $\partial_F(\G)$ is faithful then $\G$ has the unique trace property. Both the converse and that simplicity of $C_r(\widehat{\G})$ implies the unique trace property were proved recently in \cite{ASK23}.

With \cite[Lemma 5.2]{KKSV22}, it was shown that the $\G$-invariant tracial states are the KMS states of the scaling automorphism group of $\widehat{\G}$ (see \cite{KKSV22}). As a consequence, we have that a state on $C_r(\widehat{\G})$ is $\G$-invariant if and only if it is tracial in the case where $\G$ is unimodular. The latter result was observed by Izumi \cite{I02} in a slightly different form. They also prove that faithfulness of $\G\acts\partial_F(\G)$ implies $C_r(\widehat{\G})$ has no $\G$-invariant states whenever $\G$ is non-unimodular.

In this paper, we show that the above characterization of tracial states from the unimodular case holds in the general case as well, despite the fact that the scaling automorphism group may not be trivial.\\
\\
{\bf Theorem \ref{traceiffGinv}} {\it Let $\G$ be a DQG. A state $\tau\in C_r(\widehat{\G})^*$ is tracial if and only if it is $\G$-invariant.}\\
\\
In particular, combined with \cite[Lemma 5.2]{KKSV22}, Theorem \ref{traceiffGinv} establishes a bijection between the following sets of states on $C_r(\widehat{\G})$:
\begin{align*}
    \text{tracial states}~ \iff ~\text{$\G$-invariant states} ~\iff ~\text{$\tau$-KMS states}
\end{align*}
where $(\tau_t)_{t\in \R}$ is the scaling autmorphism group of $\widehat{\G}$.

As an easy consequence, using the work of Kalantar et al. discussed above, we have that faithfulness of $\G\acts\partial_F(\G)$ implies $C_r(\widehat{\G})$ has no tracial states whenever $\G$ is non-unimodular.

It was proved in \cite{L73} that $C_r(\widehat{G})$ is nuclear if and only if $G$ is amenable. This result was generalized to locally compact groups in \cite{N15} where it was shown that a locally compact group $G$ is amenable if and only if $C_r(\widehat{G})$ is nuclear and admits a tracial state. It seemed that what was essential to the classical proofs was the $G$-invariance of the tracial states. Indeed, a consequence of the main results in \cite{NV17} and Tomatsu's result that amenability of $\G$ is equivlalent to coamenability of $\widehat{\G}$ \cite{Toma}, was that $C_r(\widehat{\G})$ is nuclear and admits a $\G$-invariant state if and only if $\widehat{\G}$ is coamenable, where $\G$ is a discrete quantum group. This is both a quantum group dynamical and operator algebraic characterization of amenability. This statement was then generalized to locally compact quantum groups in \cite{C19}. With Theorem \ref{traceiffGinv}, we demonstrate that we can replace $\G$-invariance with traciality in the quantum version of the above results.\\
\\
{\bf Corollary \ref{nuclear}}. {\it Let $\G$ be a DQG. We have that $C_r(\widehat{\G})$ is nuclear and has a tracial state if and only if $\G$ is amenable.}\\
\\
In particular, Corollary \ref{nuclear} establishes a purely operator algebraic characterization of amenability of $\G$ (in terms of nuclearity).

A new proof of the fact that if $C_r(\widehat{G})$ is nuclear and admits a tracial state then $G$ is amenable for a locally compact group $G$ was found in \cite{FSW22}. Inspired by their techniques, we provide a second proof of Corollary \ref{nuclear} that avoids Theorem \ref{traceiffGinv}. With Proposition \ref{tracepropertiesDQGs}, we prove that $C_r(\widehat{\G})$ admits a tracial state if and only if $\widehat{\G}/\widehat{\G}_{Kac}$ is coamenable. This result, in combination with permanence properties of both nuclearity and coamenability and a reduction to the case of unimodular discrete quantum groups, allows us to obtain a proof of Corollary \ref{nuclear}.

It is straightforward to prove that if $C_r(\widehat{\G})$ is simple then it has no tracial states whenever $\G$ is non-unimodular. As a consequence of Theorem \ref{traceiffGinv}, we also obtain that simplicity of $C_r(\widehat{\G})$ implies the non-existence of $\G$-invariant states for non-unimodular $\G$ as well.

So, we have established that the $\G$-invariant states are exactly the tracial states on the reduced $C^*$-algebras of arbitrary DQGs. It is this and our next result that comprises the rest of the key results of this paper.\\
\\
{\bf Theorem \ref{concentratesonrad}.} {\it Let $\G$ be a DQG. Every tracial state $\tau\in C_r(\widehat{\G})^*$ concentrates on $\widehat{\G}/\widehat{\G}_F$.}\\
\\
We must note that most of the above theorem was proven in the proof of \cite[Theorem 5.3]{KKSV22} (when we replace traciality with $\G$-invariance). Our main contribution for this result is Theorem \ref{traceiffGinv}, which allows it to be stated for tracial states.

Besides this, we analyze how the existence and uniqueness of traces relates to the canonical Kac quotient $\widehat{\G}_{Kac}$ in Proposition \ref{tracepropertiesDQGs}. An essential piece of our work is the following.\\
\\
{\bf Proposition \ref{Canonical Kac Quotient is Maximal} and \ref{Coamenable Kac Quotient}}
\begin{enumerate}
    \item We have that $\widehat{\G}_{MaxKac} = \widehat{\G}_{Kac}$.

    \item Suppose $C_r(\widehat{\G})$ admits a tracial state. Then $\widehat{\G}_{rKac} = \widehat{\G}_{Kac}$ and $\widehat{\G} / \widehat{\G}_{rKac}$ is coamenable.
\end{enumerate}
Above, $\widehat{\G}_{MaxKac}$ is the maximal Kac type closed quantum subgroup of $\widehat{\G}$ (\cite[Theorem 5.9]{Ch15}, \cite{T07, FFS23}), $\widehat{\G}_{Kac}$ is the CQG underlying So\l tan's canonical Kac quotient of $C_u(\widehat{\G})$, and $\widehat{\G}_{rKac}$ is the CQG underlying So\l tan's canonical Kac quotient of $C_r(\widehat{\G})$. It appears to be folklore among experts that $\widehat{\G}_{MaxKac} = \widehat{\G}_{Kac}$ (see \cite[Section 2.3]{NY16}), however, we are not aware of a proof in the literature, so we give one in this paper. Our claim regarding $\widehat{\G}_{rKac}$ does appear to be new.

By combining the above results with our work in Proposition \ref{tracepropertiesDQGs}, we obtain the following.\\
\\
{\bf Corollary \ref{existenceanduniqueoftraces}}. {\it Let $\G$ be a DQG. Suppose that $\widehat{\G}/\widehat{\G}_F$ is coamenable. The following hold:
    \begin{enumerate}
        \item $\widehat{\G}_F$ is Kac type if and only if $C_r(\widehat{\G})$ has a tracial state;
        
        \item  $C_r(\widehat{\G})$ has a tracial state, $\widehat{\G}_F = \widehat{\G}_{Kac}$, and $C_{rKac}(\widehat{\G}_{Kac}) = C_r(\widehat{\G}_F)$ if and only if $C_r(\widehat{\G})$ has a unique tracial state;
        
        \item $C_r(\widehat{\G})$ has a tracial state and $\widehat{\G}_F = \widehat{\G}_{Kac}$ if and only if $C_r(\widehat{\G})$ has a unique idempotent tracial state.
    \end{enumerate}}
See Section $2.4$ for coamenability of quotients. We note that $\widehat{\G}_{Kac}$ is the canonical Kac quotient of $\G$, as defined in \cite{S06} (see Section $3.1$). We also note that it is currently an open question to determine whether or not $\widehat{\G}/\widehat{\G}_F$ is coamenable in general (see \cite[Question 8.1]{KKSV22}). Significant progress has been made recently, however. In the paper \cite{ASK23} a positive answer was established in the cases where $\G$ is unimodular and where $\G$ is exact.

So, it is the Haar state coming from the cokernel of $\G\acts \partial_F(\G)$ that governs the existence and uniqueness of traces of discrete quantum groups (at least when it is itself reduced, i.e., factors through the reduced $C^*$-algebra).

We discuss now the organization of this paper. Section $2$ is reserved for preliminary concepts. We discuss locally compact quantum groups, $\G$-boundaries, closed quantum subgroups, idempotent states, and (co)amenability. We reserve Section $3$ for our main theorems. We recall the construction of the Furstenberg boundary, the Kac and unimodularity properties, the construction of the canonical Kac quotient, and the basics of $\G$-invariant states. We spend the remainder of the paper proving our main theorems highlighted above.

\section{Basics of Quantum Groups}
\subsection{Discrete Quantum Groups}
The notion of a quantum group we will be using is the von Neumann algebraic one developed by Kustermans and Vaes \cite{KV00}. A {\bf locally compact quantum group} (LCQG) $\G$ is a quadruple $(L^\infty(\G), \Delta_\G, h_L, h_R)$ where $L^\infty(\G)$ is a von Neumann algebra; $\Delta_\G : L^\infty(\G) \to L^\infty(\G)\overline{\otimes}L^\infty(\G)$ a normal unital $*$--homomorphism satisfying $(\Delta_\G\otimes\id)\circ\Delta_\G = (\id\otimes\Delta_\G)\circ\Delta_\G$ (coassociativity); and $h_L$ and $h_R$ are normal semifinite faithful weights on $L^\infty(\G)$ satisfying
$$h_L(f\otimes \id)\Delta_\G(x) =f(1)h_L(x), ~ 0\leq f\in L^1(\G), 0\leq x\in \mathcal{M}_{h_L} ~ \text{(left invariance)}$$
and
$$h_R(\id\otimes f)\Delta_\G(x) = f(1)h_R(x), ~ 0\leq f\in L^1(\G), 0\leq x\in \mathcal{M}_{h_R} ~ \text{(right invariance)},$$
where $\mathcal{M}_{h_L}$ and $\mathcal{M}_{h_R}$ are the set of integrable elements of $L^\infty(\G)$ with respect to $h_L$ and $h_R$ respectively. We call $\Delta_\G$ the {\bf coproduct} and $h_L$ and $h_R$ the {\bf left and right Haar weights} respectively, of $\G$. The predual $L^1(\G) := L^\infty(\G)_*$ is a Banach algebra with respect to the product $f*g:= (f\otimes g)\circ\Delta_\G$ known as {\bf convolution}.

Using $h_L$, we can build a GNS Hilbert space $L^2(\G)$ in which $L^\infty(\G)$ is standardly represented. There is a unitary $W_\G\in L^\infty(\G)\overline{\otimes}\fB(L^2(\G))$ such that $\Delta_\G(x) = W_\G^*(1\otimes x)W_\G$. The unitary $W_\G$ is known as the {\bf left fundamental unitary} of $\G$. The {\bf left regular representation} is the representation
$$\lambda_\G : L^1(\G)\to \fB(L^2(\G)), ~ f\mapsto (f\otimes\id)(W_\G).$$
There is a dense involutive algebra $L^1_\#(\G)\subseteq L^1(\G)$ that makes $\lambda_\G|_{L^{\#}(\G)}$ a $*$-representation. The {\bf unitary antipode} is the $*$-antiautomorphism
$$R_\G : L^\infty(\G)\to L^\infty(\G), ~ x\mapsto J_{\widehat{\G}}x^*J_{\widehat{\G}},$$
where $J_{\G} : L^2(\G)\to L^2(\G)$ is the modular conjugation in the standard form of $L^\infty(\G)$. The {\bf unitary antipode} satisfies the formula $(R_\G\otimes R_\G)\circ\Delta_\G = \sigma\circ\Delta_\G\circ R_\G$, where $\sigma$ is the flip, which satisfies $\sigma(a\otimes b) = b\otimes a$. The {\bf antipode} is the antimultiplicative linear map $S_\G(x) = \tau_{-i/2}\circ R_\G(x)$, $x\in D(S_\G)$, where $(\tau_t)_{t\in \R}$ is the scaling group of $\G$. Then, $L^1_\#(\G) = \{f\in L^1(\G) : \overline{f}\circ S_\G\in L^1(\G)\}$
and the involution on $L^1_\#(\G)$ is $f\mapsto f^\#$ where $f^\#(x) = \overline{f(S_\G(x)^*)}$.

We denote the von Neumann algebra $L^\infty(\widehat{\G}) = \lambda_\G(L^1(\G))''$. There exists a LCQG $\widehat{\G} = (L^\infty(\widehat{\G}),\Delta_{\widehat{\G}}, \widehat{h_L}, \widehat{h_R})$, where $\Delta_{\widehat{\G}}$ is implemented by $W_{\widehat{\G}} = \Sigma W_\G^* \Sigma$, where $\Sigma : a\otimes b\mapsto b\otimes a$ is the flip map. Pontryagin duality holds: $\hat{\hat{\G}} = \G$.

A {\bf discrete quantum group (DQG)} is a LCQG $\G$ where $L^1(\G) ~(= \ell^1(\G))$ is unital (cf. \cite{R08}, \cite{Wor}). Note also that we write $\ell^\infty(\G) = L^\infty(\G)$. We denote the unit by $\epsilon_\G$. Equivalently, $\widehat{\G}$ is a {\bf compact quantum group (CQG)}, which is a LCQG where $h_{\widehat{\G}} = \widehat{h_L} = \widehat{h_R}\in L^1(\widehat{\G})$ is a state, known as the Haar state of $\widehat{\G}$. When $\G$ is discrete, the irreducible $*$-representations of $L^1(\widehat{\G})$ are finite dimensional, where a $*$-representation of a locally compact $\G$ is a non-degenerate cb-representation that restricts to a $*$-representation on $L^1_\#(\G)$. Given a cb-representation $\pi : L^1(\widehat{\G})\to \fB(\fH_\pi)\cong M_{n_\pi}$, there exists an operator $U^\pi\in L^\infty(\widehat{\G})\overline{\otimes}\fB(\fH_\pi)$ such that
$$\pi(u) = (u\otimes \id)(U^\pi), ~ u\in L^1(\widehat{\G}).$$
Equivalently, a cb-representation $L^1(\widehat{\G})\to \fB(\fH_\pi)$ is a $*$-representation if $U^\pi$ is unitary. Representations $\pi$ and $\rho$ are unitarily equivalent if there exists a unitary $U\in M_{n_\pi}$ such that $(1\otimes U^*)U^\pi(1\otimes U) = U^\rho$. We let $Irr(\widehat{\G})$ denote the collection of equivalence classes of irreducibles. Note that we will abuse notation and simply write $\pi\in Irr(\widehat{\G})$ when we are choosing a representative from $[\pi]$. For each $\pi\in Irr(\widehat{\G})$ we let $\fH_\pi$ denote the corresponding $n_\pi$-dimensional Hilbert space. In the instance where $\pi\in Irr(\widehat{\G})$, we write $U^\pi = [u_{i,j}^\pi]_{i,j}$, so $\pi(u) = [u(u_{i,j}^\pi)]_{i,j}$ for $u\in L^1(\widehat{\G})$ and some orthonormal basis (ONB) $\{e_i^\pi\}$ of $\fH_\pi$. We let $\overline{\pi} : L^1(\widehat{\G})\to \fB(\fH_{\overline{\pi}})$ denote the representation where $U^{\overline{\pi}} = [(u_{i,j}^\pi)^*] = \overline{U^\pi}$.

It turns out that every $*$-representation decomposes into a direct sum of irreducibles and the left regular representation decomposes into a direct sum of elements from $Irr(\widehat{\G})$, each with multiplicity $n_\pi$. Consequently, we have $W_{\widehat{\G}} = \bigoplus_{\pi\in Irr(\widehat{\G})} U^\pi\otimes I_{n_{\overline{\pi}}}$ using the identification $L^2(\widehat{\G}) = \oplus_{\pi\in Irr(\widehat{\G})}\fH_\pi\otimes \fH_{\overline{\pi}}$.

We denote the $*$-algebra
$$\Pol(\widehat{\G}) = \Span\{ u^\pi_{i,j} : 1\leq i,j\leq n_\pi, \pi\in Irr(\widehat{\G}))\}\subseteq L^\infty(\widehat{\G}).$$
It follows that $L^\infty(\widehat{\G}) = \overline{\Pol(\widehat{\G})}^{wk*}$ from Pontryagin duality.

Focusing on $\G$, the above entails that we have $\ell^\infty(\G) = \bigoplus^{\ell^\infty}_{\pi\in Irr(\widehat{\G})} M_{n_\pi}$ as von Neumann algebras, which implies $\ell^1(\G) = \bigoplus^{\ell^1}_{\pi\in Irr(\widehat{\G})}(M_{n_\pi})_*$ spatially. We denote the $C^*$-algebra $c_0(\G) = \bigoplus_{\pi\in Irr(\widehat{\G})}^{c_0}M_{n_\pi}$, which we note satisfies $M(c_0(\G)) = \ell^\infty(\G)$.

For this discussion, we are required to discuss the $C^*$-algebraic formulation of quantum groups in the universal setting (cf. \cite{K01}). There exists a universal $C^*$-norm $||\cdot||_u$ on $\Pol(\widehat{\G})$. Let $||\cdot||_r$ be the norm on $\fB(L^2(\widehat{\G}))$. We define the unital $C^*$-algebras $C_u(\widehat{\G}) = \overline{\Pol(\G)}^{||\cdot||_u}$ and $C_r(\widehat{\G}) = \overline{\Pol(\widehat{\G})}^{||\cdot||_r} \subseteq L^\infty(\widehat{\G})$. The universal property gives us a $C^*$-algebraic coproduct on $C_u(\widehat{\G})$: a unital $*$-homomorphism
$$\Delta^u_\G : C_u(\widehat{\G})\to C_u(\widehat{\G})\otimes_{min}C_u(\widehat{\G})$$
satisfying coassociativity. Likewise, $\Delta^r_{\widehat{\G}} = \Delta_{\widehat{\G}}|_{C_r(\widehat{\G})}$ gives us a $C^*$-algebraic coproduct on $C_r(\widehat{\G})$. The spaces $C_u(\widehat{\G})^*$ and $C_r(\widehat{\G})^*$, and are known as the {\bf universal and reduced measure algebras} of $\widehat{\G}$ respectively. Similar to the von Neumann algebraic case, the coproduct on $C(\widehat{\G})$ induces a product on $C(\widehat{\G})^*$:
$$\mu*\nu(a) = (\mu\otimes\nu)\Delta(a), ~ a\in C_u(\widehat{\G}), \mu,\nu \in C(\widehat{\G})^*,$$
making $C(\widehat{\G})^*$ a Banach algebra, where, above, $C(\widehat{\G})$ and $C(\widehat{\G})^*$ can be either the universal or reduced versions.

The universal property gives us a unital surjective $*$-homomorphism
$$\Gamma_{\widehat{\G}} : C_u(\widehat{\G})\to C_r(\widehat{\G})$$
satisfying $\Delta_{\widehat{\G}}^r\circ \Gamma_{\widehat{\G}} = (\Gamma_{\widehat{\G}}\otimes\Gamma_{\widehat{\G}})\circ\Delta^u_{\widehat{\G}}$. The adjoint of this map induces a completely isometric algebra homomorphism $C_r(\widehat{\G})^*\to C_u(\widehat{\G})^*$ such that $C_r(\widehat{\G})^*$ is realized as a weak$^*$ closed ideal in $C_u(\widehat{\G})^*$. A Hahn-Banach argument shows $\overline{L^1(\widehat{\G})}^{wk*} = C_r(\widehat{\G})^*$, and $L^1(\widehat{\G})$ is a closed ideal in $C_u(\widehat{\G})^*$ as well.

More generally, we say $C_\sigma(\widehat{\G})$ is a {\bf CQG $C^*$-algebra} if it is a $C^*$-algebraic completion of $\Pol(\widehat{\G})$ and admits a coproduct $\Delta_\sigma$ such that $\Delta_\sigma|_{\mathrm{Pol}(\widehat{\G})} = \Delta|_{\mathrm{Pol}(\widehat{\G})}$. There always exists canonical quotient maps $C_u(\widehat{\G})\to C_\sigma(\widehat{\G})\to C_r(\widehat{\G})$ that intertwine the coproducts such as $\Gamma_{\widehat{\G}}$ does for the universal and reduced coproducts.
\begin{defn}
    A CQG $C^*$-algebra $C_\sigma(\widehat{\G})$ is {\bf exotic} if the canonical quotient maps $C_u(\widehat{\G})\to C_\sigma(\widehat{\G})\to C_r(\widehat{\G})$ are non-injective.
\end{defn}
The restriction of the coproduct is a unital $*$-homomorphism $\Delta_{\widehat{\G}} : \Pol(\widehat{\G})\to \Pol(\widehat{\G})\otimes \Pol(\widehat{\G})$ that satisfies $\Delta_{\widehat{\G}}(u_{i,j}^\pi) = \sum_{t=1}^{n_\pi}u_{i,t}^\pi\otimes u_{t,j}^\pi$. The unital $*$-homomorphism $\epsilon_{\widehat{\G}} : \Pol(\widehat{\G})\to\C$, $u_{i,j}^\pi\mapsto \delta_{i,j}$
extends to a unital $*$-homorphism $\epsilon^u_{\widehat{\G}} : C_u(\widehat{\G})\to \C$ which is the identity element in $C_u(\widehat{\G})^*$. Finally, if $u_{i,j}^\pi \neq 1$, then $h_{\widehat{\G}}(u_{i,j}^\pi) = 0$.

\begin{rem}\label{CommutativeCase}
    The examples of DQGs where $\ell^\infty(\G)$ is commutative are the discrete groups (cf. \cite{T69}), where if $G$ is a discrete group, then $G = (\ell^\infty(G), \Delta_G, m_L, m_R)$ where $m_L = m_R = h_L = h_R$ is the counting measure, and $\Delta_G(f)(s,t) = f(st)$.
\end{rem}

\subsection{Closed Quantum Subgroups}
We use \cite{Daws} and \cite{Kal1} as our primary reference for closed quantum subgroups. We use the notion of a quantum subgroup in the sense of Woronowicz. Let $\G$ and $\H$ be DQGs. We will say $\widehat{\H}$ is a closed quantum subgroup of $\widehat{\G}$ if there exists a surjective unital $*$-homomorphism $\pi_{\widehat{\H}} : \Pol(\widehat{\G})\to \Pol(\widehat{\H})$ such that $(\pi_{\widehat{\H}}\otimes\pi_{\widehat{\H}})\circ\Delta_{\widehat{\G}} = \Delta_{\widehat{\H}}\circ\pi_{\widehat{\H}}$. We will write $\widehat{\H}\leq \widehat{\G}$. In the sequel, we will need to pass between the ``algebraic analogue'' of the map $\pi_{\widehat{\H}}$ as it is defined above and the ``operator algebraic analogues'' of such a map. More precisely, we will frequently use the following lemma.
\begin{lem}\label{Passing CQGs}\cite[Lemma 2.8]{T07}
    \begin{enumerate}
        \item Let $C_{\sigma_1}(\widehat{\G})$ and $C_{\sigma_2}(\widehat{\H})$ be (possibly exotic) CQG $C^*$-algebras of $\widehat{\G}$ and $\widehat{\H}$ respectively. Every surjective unital $*$-homomorphism $\pi_{\widehat{\H}}^{\sigma_{1,2}} : C_{\sigma_1}(\widehat{\G})\to C_{\sigma_2}(\widehat{\H})$ such that $(\pi_{\widehat{\H}}^{\sigma_{1,2}}\otimes \pi_{\widehat{\H}}^{\sigma_{1,2}})\circ\Delta^{\sigma_1}_{\widehat{\G}} = \Delta^{\sigma_2}_{\widehat{\H}}\circ\pi_{\widehat{\H}}^{\sigma_{1,2}}$ restricts to a surjective unital $*$-homorphism $\pi_{\widehat{\H}}^{\sigma_{1,2}}|_{\Pol(\widehat{\G})} = \pi_{\widehat{\H}} : \Pol(\widehat{\G})\to \Pol(\widehat{\H})$ such that $(\pi_{\widehat{\H}}\otimes \pi_{\widehat{\H}})\circ\Delta_{\widehat{\G}} = \Delta_{\widehat{\H}}\circ\pi_{\widehat{\H}}$.

        \item Every surjective unital $*$-homomorphism $\pi_{\widehat{\H}} : \Pol(\widehat{\G})\to \Pol(\widehat{\H})$ such that $(\pi_{\widehat{\H}}\otimes\pi_{\widehat{\H}})\circ\Delta_{\widehat{\G}} = \Delta_{\widehat{\H}}\circ\pi_{\widehat{\H}}$ extends to a surjective unital $*$-homomorphism $\pi_{\widehat{\H}}^u : C_u(\widehat{\G})\to C_u(\widehat{\H})$ such that $(\pi_{\widehat{\H}}^u\otimes\pi_{\widehat{\H}}^u)\circ\Delta^u_{\widehat{\G}} = \Delta^u_{\widehat{\H}}\circ\pi^u_{\widehat{\H}}$.
    \end{enumerate}
\end{lem}
The proof of the above lemma follows from the same argument used in the proof of \cite[Lemma 2.8]{T07}. We emphasize that the results as stated above do not require the Haar states to be faithful (see the first line after \cite[Definition 2.1]{T07}) or the CQGs to be coamenable (as stated in part $2$ of \cite[Lemma 2.8]{T07}).

In the setting of CQGs, every closed quantum subgroup in the sense of Woronowicz is a closed quantum subgroup in the sense of Vaes. In particular, if $\widehat{\H}\leq \widehat{\G}$ then there is an injective normal unital $*$-homomorphism $\gamma_{\H} : \ell^\infty(\H)\to \ell^\infty(\G)$ such that
$$(\gamma_{\H}\otimes\gamma_{\H})\circ\Delta_{\H} = \Delta_\G\circ\gamma_{\H}.$$
Conversely, if $\H$ and $\G$ are DQGs such that there is a normal unital $*$-homomorphism $\gamma_{\H} : \ell^\infty(\H)\to \ell^\infty(\G)$ as above, then $\widehat{\H}\leq \widehat{\G}$ as well.
%\begin{eg}
    %Let $\G = G$ be a discrete group. We have that $\H\leq \widehat{G}$ is a closed quantum subgroup if and only if $\H = \widehat{G/N}$ where $N$ is a normal subgroup of $G$. 
%\end{eg}
Given $\widehat{\H}\leq\widehat{\G}$, we define the (left) quotient space:
$$\Pol(\widehat{\G}/\widehat{\H}) := \{a\in \Pol(\widehat{\G}) : (\id\otimes\pi_{\widehat{\H}})(\Delta_{\widehat{\G}}(a)) = a\otimes 1\}.$$
Analogously, we can define the right quotient space $\Pol(\widehat{\H}\backslash\widehat{\G})$ in the obvious way. We denote the reduced completions by
$$L^\infty(\widehat{\G}/\widehat{\H}) = \Pol(\widehat{\G}/\widehat{\H})'' ~ \text{and} ~ C_r(\widehat{\G}/\widehat{\H}) = \overline{\Pol(\widehat{\G}/\widehat{\H})}\subseteq C_r(\widehat{\G})$$
and the universal completion by
$$C_u(\widehat{\G}/\widehat{\H}) = \overline{\Pol(\widehat{\G}/\widehat{\H})}\subseteq C_u(\widehat{\G}).$$
It is straightforward to see that $\Gamma_{\widehat{\G}}(C_u(\widehat{\G}/\widehat{\H})) = C_r(\widehat{\G}/\widehat{\H})$ where $\Gamma_{\widehat{\G}} : C_u(\widehat{\G})\to C_r(\widehat{\G})$ is the canonical quotient morphism.
\begin{rem}
    In \cite[p. 801 and Definition 3.8]{KKSV22}, the quotient space $\widehat{\G}/\widehat{\H}$ was defined by setting $C_u(\widehat{\G}/\widehat{\H}) = \{a\in C_u(\widehat{\G}) : (\id\otimes\pi^u_{\widehat{\H}})\circ\Delta^u_{\widehat{\G}}(a) = a\otimes 1\}$. Then, $C_r(\widehat{\G}/\widehat{\H})$ was defined by setting $C_r(\widehat{\G}/\widehat{\H}) := \Gamma_{\widehat{\G}}(C_u(\widehat{\G}/\widehat{\H}))$. It is well-known that this is an equivalent way of defining $\widehat{\G}/\widehat{\H}$. Indeed,
    $$\Pol(\widehat{\G}/\widehat{\H}) = \{a\in \Pol(\widehat{\G}) : (\id\otimes h_{\widehat{\H}}\circ\pi_{\widehat{\H}})(\Delta_{\widehat{\G}}(a)) = a\}.$$
    Then, according to, for example, the discussion in the paragraphs above \cite[Proposition 3.2]{Wang1}, $\Pol(\widehat{\G})\cap C_u(\widehat{\G}/\widehat{\H}) = \Pol(\widehat{\G}/\widehat{\H})$ and $\Pol(\widehat{\G}/\widehat{\H})$ is norm dense in $C_u(\widehat{\G}/\widehat{\H})$. The same argument can be used to prove that our definition of $L^\infty(\widehat{\G}/\widehat{\H})$ coincides with that of \cite{KKSV22}.
\end{rem}

\subsection{Idempotent States}
Let $\G$ be a DQG. We say a state $\omega\in C_u(\widehat{\G})^*$ is {\bf idempotent} if $\omega*\omega = \omega$. Whenever $\widehat{\H}\leq\widehat{\G}$ is a closed quantum subgroup, $h_{\widehat{\H}}\circ\pi_{\widehat{\H}} \in C_u(\widehat{\G})^*$ is an idempotent state. Not every idempotent state has this form, even for groups.
\begin{defn}
    Let $\G$ be a DQG. An idempotent state of the form $\omega_{\widehat{\G}/\widehat{\H}} = h_{\widehat{\H}}\circ\pi_{\widehat{\H}}^u$, where $\widehat{\H}$ is a closed quantum subgroup of $\widehat{\G}$, is known as a {\bf Haar idempotent}.
\end{defn}
\begin{rem}\label{Remark on Idempotent States}
\begin{enumerate}
    \item Let $\G = G$ be a discrete group. We have that $\omega\in C_u(\widehat{G})^*$ is an idempotent state if and only if $\omega = 1_H$ for some subgroup $H\leq G$ where
    $$1_H(\lambda(s)) = \begin{cases}
        1 &\text{if} ~s\in H\\ 0 &\text{otherwise}
    \end{cases}$$
    \cite[Theorem~2.1]{IS05}. It turns out that $\omega$ is Haar type if and only if $H$ is normal (this more or less follows from Theorem \ref{HaarIdempotents}).
    
    Note that in the case of a discrete group $G$ together with a normal subgroup $N$, $C_u(\widehat{G} / \widehat{G/N}) = C_u(\widehat{N})$.

    \item Let $\widehat{\H}$ be a closed quantum subgroup of $\widehat{\G}$ such that there exists a surjective unital $*$-homomorphism $\pi^\sigma_{\widehat{\H}} : C_{\sigma_1}(\widehat{\G})\to C_{\sigma_2}(\widehat{\H})$ such that $(\pi^{\sigma_2}_{\widehat{\H}}\otimes \pi^{\sigma_2}_{\widehat{\H}})\circ \Delta^{\sigma_1}_{\widehat{\G}} = \Delta^{\sigma_2}_{\widehat{\H}}\circ \pi^{\sigma_2}_{\widehat{\H}}$ where $C_{\sigma_1}(\widehat{\G})$ and $C_{\sigma_2}(\widehat{\H})$ are potentially exotic CQG $C^*$-algebras of $\widehat{\G}$ and $\widehat{\H}$ respectively (see Lemma \ref{Passing CQGs}). The adjoint induces completely isometric algebra embeddings $C_{\sigma_2}(\widehat{\H})^*\subseteq C_{\sigma_1}(\widehat{\G})^*\subseteq C_u(\widehat{\G})^*$. By definition, the Haar idempotent induced by $\widehat{\H}$ is simply given by $\omega_{\widehat{\G}/\widehat{\H}} = h_{\widehat{\H}}$ via these embeddings, where $h_{\widehat{\H}}$ is the Haar state on $C_{\sigma_2}(\widehat{\H})$.
\end{enumerate}
\end{rem}
In general, the Haar idempotents can be distinguished with the following important theorem.
\begin{thm}\cite[Theorem 5]{SS16}\label{HaarIdempotents}
    Let $\G$ be a DQG and $\omega\in C_u(\widehat{\G})^*$ an idempotent state. Then $\omega$ is a Haar idempotent if and only if $I_\omega = \{a\in C_u(\widehat{\G}) : \omega(a^*a) = 0\}$ is a two-sided ideal (equivalently, $I_\omega$ is self-adjoint).
\end{thm}
An important consequence of Theorem \ref{HaarIdempotents} for us will be that the tracial idempotent states are automatically Haar idempotents. Note that given a closed quantum subgroup $\widehat{\H}\leq \widehat{\G}$, we have that $C_r(\widehat{\H})= C_u(\widehat{\G})/I_{\omega_{\widehat{\G}/\widehat{\H}}}$ and the induced quotient map $q : C_u(\widehat{\G})\to C_u(\widehat{\G})/I_{\omega_{\widehat{\G}/\widehat{\H}}}$ satisfies $q = \Gamma_{\widehat{\H}}\circ\pi^u_{\widehat{\H}}$ (see the proof of \cite[Theorem 5]{SS16}). 

Before proceeding to the next subsection we will note some important observations for our work (see \cite{SS16}). Consider the map $L^u_{\omega_{\widehat{\G}/\widehat{\H}}} = (\id\otimes\omega_{\widehat{\G}/\widehat{\H}})\circ\Delta^u_{\widehat{\G}}$. It is known that $L^u_{\omega_{\widehat{\G}/\widehat{\H}}}$ is a conditional expectation onto $C_u(\widehat{\G}/\widehat{\H})$ and admits a reduced version
$$L^r_{\omega_{\widehat{\G}/\widehat{\H}}} : C_r(\widehat{\G})\to C_r(\widehat{\G}/\widehat{\H})$$
which satisfies $\Gamma_{\widehat{\G}}\circ L^u_{\omega_{\widehat{\G}/\widehat{\H}}}  = L^r_{\omega_{\widehat{\G}/\widehat{\H}}}\circ\Gamma_{\widehat{\G}}$. There is also a von Neumann algebraic version
$$L_{\omega_{\widehat{\G}/\widehat{\H}}} : L^\infty(\widehat{\G})\to L^\infty(\widehat{\G}/\widehat{\H})$$
such that $L_{\omega_{\widehat{\G}/\widehat{\H}}}|_{C_r(\widehat{\G})} = L^r_{\omega_{\widehat{\G}/\widehat{\H}}}$.

Let $\widehat{\H}\leq\widehat{\G}$. We set $P_\H = \lambda_{\widehat{\G}}^u(\omega_{\widehat{\G}/\widehat{\H}})$, where $\lambda^u_{\widehat{\G}}(\mu) = (\mu\otimes \id)(\mathbb{W}_{\widehat{\G}}) \in \ell^\infty(\G)$ and  $\mathbb{W}_{\widehat{\G}}$ is the half-lifted version of $W_{\widehat{\G}}$. It turns out that $P_\H\in \ell^\infty(\H)$ is the (unique) minimal group-like projection that generates $\ell^\infty(\H)$ (see the proof of \cite[Theorem 3.1]{FK17}), i.e., $P_\H$ is a self-adjoint projection such that $(1\otimes P_\H)\Delta_\G(P_\H) = P_\H\otimes P_\H$ and
$$\ell^\infty(\H) = \{x\in \ell^\infty(\G) : \Delta_\G(x)(1\otimes P_\H) = x\otimes P_\H\}.$$
Given $\widehat{\K},\widehat{\H}\leq \widehat{\G}$, it follows from minimality that $P_\K P_\H  = P_\H$ if and only if $\widehat{\K}\leq \widehat{\H}$ - note that in this situation $\ell^\infty(\K)\subseteq\ell^\infty(\H)$ via the embeddings $\gamma_{\K}$ and $\gamma_{\H}$ into $\ell^\infty(\G)$ respectively and equivalently the map $\pi^u_{\widehat{\K}} : C_u(\widehat{\G})\to C_u(\widehat{\K})$ factors through the canonical quotient $\pi^u_{\widehat{\H}} : C_u(\widehat{\G})\to C_u(\widehat{\H})$: we have $\pi^u_{\widehat{\K}} = \pi^u_{\widehat{\H},\widehat{\K}}\circ\pi^u_{\widehat{\H}}$ where $\pi^u_{\widehat{\H},\widehat{\K}} : C_u(\widehat{\H})\to C_u(\widehat{\K})$ is the canonical quotient map. Such a claim is surely known, however, we will justify it here (cf. \cite{Soltan} and the proof of \cite[Theorem 3.1]{FK17}). Indeed, if $P_\K P_\H = P_\H$ and $x\in \ell^\infty(\K)$, then $P_\H P_\K = P_\H$ and
$(1\otimes P_\H)\Delta_\G(x) = (1\otimes P_\H P_\K)\Delta_\G(x) = x\otimes P_\H$, which implies $x\in\ell^\infty(\H)$. Conversely, it follows that $(1\otimes P_\H)\Delta_\G(P_\K) = P_\K\otimes P_\H$ and then an application of $\epsilon_\G\otimes\id$ to the above equation gives $P_\H P_\K = P_\H$ (where we are using the fact that $\epsilon_\G(P_\K) = 1$ (this can be observed by applying $\epsilon_\G\otimes\id$ to $P_\K\otimes P_\K = (1\otimes  P_\K)\Delta_\G(P_\K)$)).

More generally, the techniques we used above can be used to prove that $P_\H x = \epsilon_\G(x) P_\H$ for all $x\in \ell^\infty(\H)$ (see also the proof of \cite[Theorem 3.1]{FK17}).
\begin{rem}\label{Comparison of Quotients}
    Let $\widehat{\K}, \widehat{\H}\leq\widehat{\G}$. It is known that
    $$L^\infty(\widehat{\G}/ \widehat{\H})\subseteq L^\infty(\widehat{\G}/\widehat{\K}) \iff \widehat{\K}\leq\widehat{\H} ~ \text{(see \cite[Lemma 2.1]{Soltan})}.$$
    We will provide a proof of this statement here. It is clear that
    $$P_\K P_\H =  P_\H \iff \omega_{\widehat{\G}/\widehat{\K}}*\omega_{\widehat{\G}/\widehat{\H}} = \omega_{\widehat{\G}/\widehat{\H}}.$$
    Therefore,
    \begin{align*}
        \widehat{\K}\leq\widehat{\H} \iff\omega_{\widehat{\G}/\widehat{\K}}*\omega_{\widehat{\G}/\widehat{\H}} = \omega_{\widehat{\G}/\widehat{\H}}&\iff L_{\omega_{\widehat{\G}/\widehat{\K}}}\circ L_{\omega_{\widehat{\G}/\widehat{\H}}} = L_{\omega_{\widehat{\G}/\widehat{\K}}*\omega_{\widehat{\G}/\widehat{\H}}} = L_{\omega_{\widehat{\G}/\widehat{\H}}}
        \\
        &\iff L^\infty(\widehat{\G}/\widehat{\H})\subseteq L^\infty(\widehat{\G}/\widehat{\K}).
    \end{align*}
\end{rem}

\subsection{Coamenability}
The adjoint of the quotient map $\Gamma_{\widehat{\G}} :C_u(\widehat{\G})\to C_r(\widehat{\G})$ is a (completely isometric) algebra inclusion
$$\Gamma_{\widehat{\G}}^* : C_r(\widehat{\G})^*\to C_u(\widehat{\G})^*, ~ \mu\mapsto \mu\circ \Gamma_{\widehat{\G}}.$$
In fact, $\Gamma_{\widehat{\G}}^*(C_r(\widehat{\G})^*)$ is a weak$^*$ closed two-sided ideal in $C_u(\widehat{\G})^*$. We will adopt the convention to identify $C_r(\widehat{\G})^*$ with its image in $C_u(\widehat{\G})^*$. To be completely clear, given an element $\mu\in C_u(\widehat{\G})^*$, when we say $\mu\in C_r(\widehat{\G})^*$ we mean that $\mu\in \Gamma_{\widehat{\G}}^*(C_r(\widehat{\G})^*)$ and there exists an element $\mu^r\in C_r(\widehat{\G})^*$ such that $\mu = \mu^r\circ\Gamma_{\widehat{\G}}$. In some situations we will use the notation $\mu^r\circ\Gamma_{\widehat{\G}}$ to denote the factorization of a functional $\mu$ on $C_u(\widehat{\G})$ through the reducing morphism $\Gamma_{\widehat{\G}}$ (when such a factorization exists).
\begin{defn}
    A CQG $\widehat{\G}$ is {\bf coamenable} if $\epsilon_{\widehat{\G}}^u\in C_r(\widehat{\G})^*$. We say $\G$ is {\bf amenable} whenever there exists a state $m\in \ell^\infty(\G)^*$ satisfying $m\circ(\id\otimes f)\circ\Delta_\G(x) = f(1)m(x)$ for all $x\in\ell^\infty(\G)$ and $f\in \ell^1(\G)$.
\end{defn}
It turns out that a DQG $\G$ is amenable if and only if $\widehat{\G}$ is coamenable. In the unimodular case this was proved by \cite{Ruan} and the general case was proved by Tomatsu \cite{Toma}.

The following definition was stated in \cite{KKSV22}.
\begin{defn}
    Let $\G$ be a DQG and $\widehat{\H}\leq\widehat{\G}$ a closed quantum subgroup of $\widehat{\G}$. We say $\widehat{\G}/\widehat{\H}$ is {\bf coamenable} if there exists a state $\epsilon^r \in C_r(\widehat{\G}/\widehat{\H})^*$ such that $\epsilon^r\circ\Gamma_{\widehat{\G}}|_{C_u(\widehat{\G}/\widehat{\H})} = \epsilon^u_{\widehat{\G}}|_{C_u(\widehat{\G}/\widehat{\H})}$.
\end{defn}
An important characterization of coamenability of a quotient is the following.
\begin{thm}\cite[Theorem 3.11]{KKSV22}\label{quotientcoamenable}
    Let $\G$ be a DQG and $\H$ a DQG such that $\widehat{\H}\leq\widehat{\G}$. We have that $\widehat{\G}/\widehat{\H}$ is a coamenable quotient if and only if there exists a $*$-homomorphism $\pi^r_{\widehat{\H}} : C_r(\widehat{\G})\to C_r(\widehat{\H})$ such that $\pi^r_{\widehat{\H}}\circ \Gamma_{\widehat{\G}} = \Gamma_{\widehat{\H}}\circ\pi_{\widehat{\H}}^u$ if and only if $\omega_{\widehat{\G}/\widehat{\H}}\in C_r(\widehat{\G})^*$.
\end{thm}
It turns out that $\Gamma_{\widehat{\H}}\circ\pi_{\widehat{\H}}^u$ is the GNS representation obtained from $\omega_{\widehat{\G}/\widehat{\H}}$ (see the proof of \cite[Theorem 3.11]{KKSV22}). So, we have that $\widehat{\G}/\widehat{\H}$ is coamenable if and only if $\omega_{\widehat{\G}/\widehat{\H}}\in C_r(\widehat{\G})^*$.

Let $G$ be a discrete group $N$ a normal subgroup of $G$ so that $\widehat{G/N}\leq \widehat{G}$. It is well-known that the quotient map $\C[G]\to \C[G/N]$ extends to a quotient map $\pi_{\widehat{G/N}}^r : C_r(\widehat{G}) \to C_r(\widehat{G/N})$ as in the statement of Theorem \ref{quotientcoamenable} if and only if $N$ is amenable. In this regard, we can view a coamenable quotient as being a quantum analogue of a normal amenable subgroup (or rather the ``codual'' analogue).
\begin{rem}\label{Hereditary Properties of Coamenability}
    Let $\widehat{\K}\leq \widehat{\H}\leq \widehat{\G}$. Suppose that $\widehat{\G}/\widehat{\K}$ is coamenable. Then $\widehat{\G}/\widehat{\H}$ is coamenable. Indeed, using Theorem \ref{quotientcoamenable} we have $\omega_{\widehat{\G}/\widehat{\K}}\in C_r(\widehat{\G})^*$. By the calculations in Remark \ref{Comparison of Quotients}, $\omega_{\widehat{\G}/\widehat{\H}} = \omega_{\widehat{\G}/\widehat{\K}}*\omega_{\widehat{\G}/\widehat{\H}} \in C_r(\widehat{\G})^*$ where we are using the fact $C_r(\widehat{\G})^*$ is an ideal in $C_u(\widehat{\G})^*$. We then deduce that $\widehat{\G}/\widehat{\H}$ is coamenable by using Theorem \ref{quotientcoamenable}.
\end{rem}
\subsection{Furstenberg Boundaries and Crossed Products}
Let $\G$ be a DQG and $A$ be a unital $C^*$-algebra.
\begin{defn}
    $A$ is a {\bf $\G$-$C^*$-algebra} if there exists a unital injective $*$-homomorphism $\alpha : A\to M(c_0(\G)\otimes_{min} A)$ satisfying
    \begin{itemize}
        \item $(\id\otimes\alpha)\circ\alpha = (\Delta_\G\otimes\id)\circ\alpha$;
        \item the closed linear span of $(c_0(\G)\otimes 1)\alpha(A)$ is norm dense in $c_0(\G)\otimes_{min}A$ (Podle\'s Condition).
    \end{itemize}
    We call $\alpha$ a (left) {\bf coaction} of $\G$ on $A$.
\end{defn}
We have that $\ell^\infty(\G)$ is a $\G$-$C^*$-algebra where the coaction is equal to the coproduct $\Delta_\G$.

For a $\G$-$C^*$-algebra $A$, will use the notation $a*f = (f\otimes\id)\alpha(a)$ for $a\in A$ and $f\in \ell^1(\G)$. Given $\G$-$C^*$-algebras $A$ and $B$, we will say a cb-map $\phi : A\to B$ is {\bf $\G$-equivariant} if for all $a\in A$ and $f\in \ell^1(\G)$, we have $\phi(a)*f = \phi(a*f)$. For any idempotent $\G$-equivariant ucp map $\phi : \ell^\infty(\G)\to \ell^\infty(\G)$, the space $\phi(\ell^\infty(\G))$ is a $\G$-$C^*$-algebra when considered as a $C^*$-algebra with the Choi-Effros product \cite[Theorem 4.9]{KKSV22}. The Choi-Effros product is defined by setting $a\cdot b := \phi(ab)$. As observed in \cite{KKSV22}, every $\G$-equivariant ucp map $A\to \ell^\infty(\G)$ is equal to the Poisson transform of some functional on $A$.
\begin{defn}
    Let $A$ be a $\G$-$C^*$-algebra. The {\bf Poisson transform} of a functional $\mu\in A^*$ is the map $\fP_\mu : A\to \ell^\infty(\G)$ defined by setting
    $$f\circ\fP_\mu(a) = \mu(a*f), ~f\in \ell^1(\G), a\in A.$$
\end{defn}
\begin{eg}\label{dualcoaction}
    The unital $C^*$-algebra $C_r(\widehat{\G})$ is a $\G$-$C^*$-algebra via the coaction $\Delta^l_\G( a) = W^*_\G(1\otimes  a)W_\G$. The coaction $\Delta^l_\G$ above is known as the {\bf adjoint coaction}.
\end{eg}
The {\bf reduced crossed product} of a $\G$-$C^*$-algebra $A$ and $\G$ is the closed linear span of $(C_r(\widehat{\G})\otimes 1)\alpha(A)$ in $M(\fK(\ell^2(\G))\otimes_{min} A)$ (see \cite{V05}). We denote the reduced crossed product of $A$ with $\G$ by $A\rtimes_r\G$. It turns out that $A\rtimes_r\G$ is a $\G$-$C^*$-algebra (see \cite[Lemma 2.11]{KKSV22} and the preceding subsections), with coaction $\beta : A\rtimes_r\G\to M(c_0(\G)\otimes_{min} A\rtimes_r\G)$ defined by setting $\beta(A) = W^*_{12} A_{23}W_{12}$. The coaction $\beta$ satisfies
\begin{itemize}
    \item $\beta|_{C_r(\widehat{\G})\otimes 1} = \Delta_\G^l\otimes \id$;
    \item $\beta|_{\alpha(A)} = \id\otimes\alpha$.
\end{itemize}
This makes the canonical embeddings of $A$ and $C_r(\widehat{\G})$ into $A\rtimes_r\G$ $\G$-equivariant.

The Furstenberg boundary was constructed for DQGs in \cite{KKSV22} with the purpose of obtaining quantum group dynamic results for DQGs. Its role was instrumental in the resolution of problems like characterizing the unique trace property and $C^*$-simplicity of reduced $C^*$-algebras of groups, as well as clarifying the relationship between the unique trace property and $C^*$-simplicity (see \cite{KK14, BKKO17}).
\begin{defn}
    We say a $\G$-$C^*$-algebra $I$ is {\bf $\G$-injective} if for any $\G$-$C^*$-algebras $A$ and $B$ with ucp $\G$-equivariant ucp maps $\psi: A\to I$ and $\iota : A\to B$, where $\iota$ is completely isometric, there exists a ucp $\G$-equivariant map $\phi : B\to I$ such that $\psi = \phi\circ\iota$.
\end{defn}
The Furstenberg boundary is constructed via \cite[Proposition 4.10, Corollary 4.11, and Proposition 4.12]{KKSV22}. It is defined by the $\G$-$C^*$-algebra $C(\partial_F\G) := \Phi_0(\ell^\infty(\G))$ where $\Phi_0$ is a minimal idempotent $\G$-equivariant ucp map in the sense of \cite[Definition 4.7]{KKSV22}. It turns out $C(\partial_F\G)$ is $\G$-injective, and if $A$ is a $\G$-injective $C^*$-algebra, then there is a $\G$-equivariant ucp embedding $C(\partial_F(\G))\subseteq A$. It also satisfies the following properties:
\begin{itemize}
    \item $\G$-essentiallity: every $\G$-equivariant ucp map $C(\partial_F\G)\to \ell^\infty(\G)$ is completely isometric;
    \item $\G$-rigidity: if $\Phi : C(\partial_F\G)\to C(\partial_F\G)$ is a $\G$-equivariant ucp map, then $\Phi = \id$;
\end{itemize}
(cf. \cite[Proposition 4.10 Proposition 4.13]{KKSV22}).

Of critical importance to us is the cokernel of the Furstenberg bounday.
\begin{defn}
    The {\bf cokernel} of $\partial_F(\G)$ is defined by the subalgebra
    $$N_F := \{\fP_\mu(a) : a\in C(\partial_F(\G)), \mu\in C(\partial_F(\G))^*\}''\subseteq\ell^\infty(\G).$$
\end{defn}
It was proved with \cite[Proposition 2.9]{KKSV22} that $N_F$ is a so-called Baaj-Vaes subalgebra of $\ell^\infty(\G)$ (see \cite[page 7]{KKSV22} for a definition). It follows from \cite[Proposition A.5]{BV05} that there is a bijection between closed quantum subgroups of $\widehat{\G}$ and Baaj-Vaes subalgebras of $\ell^\infty(\G)$ (note that this is discussed on \cite[page 7]{KKSV22}). More precisely, $N\subseteq\ell^\infty(\G)$ is a Baaj-Vaes subalgebra if and only if there exists a closed quantum subgroup $\widehat{\H}\leq\widehat{\G}$ such that $N = \ell^\infty(\H)$. In particular, there exists a closed quantum subgroup $\widehat{\G}_F\leq \widehat{\G}$, where we denote $\G_F := \widehat{\widehat{\G}_F}$, such that
$$\ell^\infty(\G_F) = N_F.$$
We will also call $\G_F$ the cokernel of $\partial_F(\G)$. We say the action of $\G$ on $\partial_F(\G)$ is {\bf faithful} when $\G_F = \G$.
\begin{rem}
    For a discrete group $G$, the kernel of the action of $G$ on $\partial_F(G)$ is $R_a(G)$, the amenable radical of $G$. Then the cokernel is $G/R_a(G)$ (cf. \cite{BKKO17}).
\end{rem}

\section{Traces on Quantum Groups}

\subsection{The Kac Property and Canonical Kac Quotient}
\begin{defn}
    A DQG $\G$ is {\bf unimodular} if $h_L = h_R$. We say $\widehat{\G}$ is {\bf Kac type} if $h_{\widehat{\G}}$ is a tracial state.
\end{defn}
Unimodularity of $\G$ is well-known to be equivalent to Kacness of $\widehat{\G}$.
\begin{thm}
    A DQG $\G$ is unimodular if and only if $\widehat{\G}$ is Kac type.
\end{thm}
From now on $\G$ is a DQG. Now, let us recall the canonical Kac quotient constructed by So\l tan \cite{S06}. Define the ideal
$$I_{Kac} = \{a \in C_u(\widehat{\G}) : \tau(aa^*) = 0 ~\text{for every tracial state} ~\tau\in C_u(\widehat{\G})^* \}.$$
The $C^*$-algebra $C_u(\widehat{\G})/ I_{Kac}$ is a CQG $C^*$-algebra. We let $\widehat{\G}_{Kac}$ denote the underlying CQG and call it the {\bf canonical Kac quotient} of $\widehat{\G}$. Note that So\l tan's construction was applied to any (possibly exotic) CQG $C^*$-algebra of $\G$ whereas in our work (and others \cite{DFS21}) we canonically apply So\l tan's construction to the universal CQG $C^*$-algebra. Also, the construction $C_u(\widehat{\G})/ I_{Kac}$ above is not necessarily universal. In lieu of Lemma \ref{Passing CQGs}, the potentially exotic setting passes to the universal setting and, due to So\l tan's work, we have that $\widehat{\G}_{Kac}$ is a closed quantum subgroup of $\widehat{\G}$. We will denote 
$$C_{uKac}(\widehat{\G}_{Kac}) = C_u(\widehat{\G}) / I_{Kac}$$
as the corresponding CQG $C^*$-algbera of $\widehat{\G}_{Kac}$. We emphasize that it is unclear if the canonical quotient map $C_u(\widehat{\G}_{Kac})\to C_{uKac}(\widehat{\G}_{Kac})$ is injective in general.
\begin{rem}
\begin{itemize}
    \item The canonical quotient $C_{uKac}(\widehat{\G}_{Kac}) \to C_r(\widehat{\G}_{Kac})$ is injective if and only if $\widehat{\G}_{Kac}$ is coamenable. Indeed, suppose $C_r(\widehat{\G}_{Kac}) = C_u(\widehat{\G})/I_{Kac}$. The counit $\epsilon_{\widehat{\G}} : C_u(\widehat{\G})\to \C$ is, in particular, a trace, hence $I_{Kac}\subseteq \ker(\epsilon_{\widehat{\G}})$. Therefore $\epsilon_{\widehat{\G}}$ factors through a $*$-homomorphism $C_r(\widehat{\G}_{Kac}) \to \C$ and hence $\widehat{\G}_{Kac}$ is coamenable by \cite[Theorem 2.8]{BML01}. The converse is obvious.

    \item Assume $\widehat{\G}$ is a Kac type CQG such that $C_u(\widehat{\G})$ is residually finite dimensional (see \cite{DFS21} for a definition). It is an easy exercise to prove that the canonical quotient map $C_u(\widehat{\G})\to C_{uKac}(\widehat{\G}_{Kac})$ is injective (for example, see \cite{DFS21}).
\end{itemize}
\end{rem}
On the other hand, there exists a Kac type quantum subgroup $\widehat{\G}_{Max Kac}\leq \widehat{\G}$ satisfying the property that $\widehat{\H}\leq \widehat{\G}_{Max Kac}$ for every Kac type quantum subgroup $\widehat{\H}\leq \widehat{\G}$ (\cite[Theorem 5.9]{Ch15}, \cite{T07, FFS23}). A priori we do not have $\widehat{\G}_{Kac} = \widehat{\G}_{MaxKac}$, but this equality holds none-the-less and appears to be folklore among experts (see \cite[Section 2.3]{NY16}).
\begin{prop}\label{Canonical Kac Quotient is Maximal}
    We have that $\widehat{\G}_{MaxKac} = \widehat{\G}_{Kac}$.
\end{prop}
\begin{proof}
    Let $q_{Kac} : C_u(\widehat{\G})\to C_{uKac}(\widehat{\G}_{Kac}) = C_u(\widehat{\G})/I_{Kac}$ be the quotient map, which was shown by So\l tan \cite{S06} to satisfy
    $$(q_{Kac}\otimes q_{Kac})\circ \Delta_{\widehat{\G}} = \Delta_{\widehat{\G}_{Kac}}^{uKac}\circ q_{Kac}.$$
    By definition, $\widehat{\G}_{Kac}\leq \widehat{\G}_{Max Kac}$. Consider the composition of quotient maps
    $$q_{Kac}^{Max Kac} : C_u(\widehat{\G}_{Max Kac}) \to C_u(\widehat{\G}_{Kac}) \to C_{uKac}(\widehat{\G}_{Kac}).$$
    The Haar state $h_{\widehat{\G}_{Max Kac}}$ is tracial, hence $h_{\widehat{\G}_{Max Kac}}\circ\pi^u_{\widehat{\G}_{Max Kac}}(a^*a) = 0$ for every $a\in I_{Kac}$. By an obvious Cauchy-Schwarz argument
    $$h_{\widehat{\G}_{Max Kac}}\circ\pi^u_{\widehat{\G}_{Max Kac}}|_{I_{Kac}} = 0.$$
    Therefore, there exists a state $h_{\widehat{\G}_{Max Kac}}^{Kac}\in C_{uKac}(\widehat{\G}_{Kac})^*$ such that $$h_{\widehat{\G}_{Max Kac}}\circ \pi^u_{\widehat{\G}_{Max Kac}} = h_{\widehat{\G}_{Max Kac}}^{Kac}\circ q_{Kac} = h_{\widehat{\G}_{Max Kac}}^{Kac}\circ q_{Kac}^{Max Kac}\circ \pi^u_{\widehat{\G}_{Max Kac}}.$$
    Hence
    $$h_{\widehat{\G}_{Max Kac}} = h_{\widehat{\G}_{Max Kac}}^{Kac}\circ q_{Kac}^{Max Kac}.$$
    Now, observe that
    \begin{align*}
        h_{\widehat{\G}_{Max Kac}} &= h_{\widehat{\G}_{Max Kac}}*(h_{\widehat{\G}_{Kac}}\circ q_{Kac}^{Max Kac})
        \\
        &= (h_{\widehat{\G}_{Max Kac}}^{Kac}\circ q_{Kac}^{Max Kac})*(h_{\widehat{\G}_{Kac}}\circ q_{Kac}^{Max Kac})
        \\
        &= (h_{\widehat{\G}_{Max Kac}}^{Kac}*h_{\widehat{\G}_{Kac}})\circ q_{Kac}^{Max Kac}
        \\
        &= h_{\widehat{\G}_{Kac}}\circ q_{Kac}^{Max Kac}
    \end{align*}
    By Lemma \ref{Passing CQGs}, the restriction $q := q_{Kac}^{Max Kac}|_{\Pol(\widehat{\G})} : \Pol(\widehat{\G})\to \Pol(\widehat{\G}_{Kac})$ is a surjective unital $*$-homomorphism and is injective because $h_{\widehat{\G}_{Max Kac}} = h_{\widehat{\G}_{Kac}}\circ q$.
\end{proof}
So\l tan's canonical Kac quotient always exists and can always be obtained from $C_u(\widehat{\G})$, as discussed above. A natural question is then, what if we attempt this construction on the reduced $C^*$-algebra of $\G$ instead? For starters, $C_r(\widehat{\G})$ may not admit a tracial state, so it might be that
$$I^r_{Kac} := \{a\in C_r(\widehat{\G}) : \tau(aa^*) = 0 ~ \text{for every tracial state} ~\tau\in C_r(\widehat{\G})\} = C_r(\widehat{\G}).$$
On the other hand, if $C_r(\widehat{\G})$ admits a trace, then there exists a Kac type CQG $\widehat{\G}_{rKac}$ that underlies $C_r(\widehat{\G}) / I_{Kac}^r$ and we set
$$C_{rKac}(\widehat{\G}_{rKac}) := C_r(\widehat{\G}) / I_{Kac}^r.$$
Like before, we can use Lemma \ref{Passing CQGs} to pass the above CQG $C^*$-algebras to the universal setting and we get that $\widehat{\G}_{rKac}\leq \widehat{\G}$. It turns out this construction still yields the canonical Kac quotient of $\widehat{\G}$ obtained in the universal setting. 
\begin{prop}\label{Coamenable Kac Quotient}
    Suppose $C_r(\widehat{\G})$ admits a tracial state. Then $\widehat{\G}_{rKac} = \widehat{\G}_{Kac}$ and $\widehat{\G} / \widehat{\G}_{rKac}$ is coamenable.
\end{prop}
\begin{proof}
    It is easily checked that the composition of quotient maps
    $$\pi^r_{\widehat{\G}_{rKac}} : C_r(\widehat{\G})\to C_{rKac}(\widehat{\G}_{rKac}) \to C_{r}(\widehat{\G}_{rKac})$$
    satisfies $(\pi^r_{\widehat{\G}_{rKac}}\otimes \pi^r_{\widehat{\G}_{rKac}})\circ\Delta_{\widehat{\G}} = \Delta_{\widehat{\G}_{rKac}}\circ \pi^r_{\widehat{\G}_{rKac}}$ and  $\pi^r_{\widehat{\G}_{rKac}}\circ \Gamma_{\widehat{\G}} = \Gamma_{\widehat{\G}_{rKac}}\circ \pi^u_{\widehat{\G}_{rKac}}$. Hence by Theorem \ref{quotientcoamenable} we have that $\widehat{\G}/\widehat{\G}_{rKac}$ is coamenable.

    Since $\widehat{\G}_{rKac}$ is Kac type, $\widehat{\G}_{rKac}\leq \widehat{\G}_{Kac}$. Then $\widehat{\G} / \widehat{\G}_{Kac}$ is coamenable thanks to Remark \ref{Hereditary Properties of Coamenability}. In particular, $\omega_{\widehat{\G}/\widehat{\G}_{Kac}} \in C_r(\widehat{\G})^*$ and hence $$I_{Kac}^r\subseteq I_{\omega_{\widehat{\G}/\widehat{\G}_{Kac}}}^r := \{a\in C_r(\widehat{\G}) : \omega_{\widehat{\G}/\widehat{\G}_{Kac}}(a^*a) = 0 \}$$
    since $\omega_{\widehat{\G}/\widehat{\G}_{Kac}}$ is tracial. Using the Cauchy Schwarz inequality, we obtain that $\omega_{\widehat{\G}/\widehat{\G}_{Kac}}$ factors through some state on $C_{rKac}(\widehat{\G}_{rKac})$ and then embeds in $S(C_u(\widehat{\G}_{rKac}))$ by taking the adjoint of the canonical quotient map $C_u(\widehat{\G}_{rKac})\to C_{rKac}(\widehat{\G}_{rKac})$, i.e, using the fact $\omega_{\widehat{\G}/\widehat{\G}_{Kac}}$ factors through $\pi^u_{\widehat{\G}_{rKac}}$. From here we can argue as in Proposition \ref{Canonical Kac Quotient is Maximal} to get that $\widehat{\G}_{rKac} = \widehat{\G}_{Kac}$.
\end{proof}

\subsection{Existence and Uninqueness of Traces and Canonical Kac Quotient}
Recall that because of Theorem \ref{HaarIdempotents}, tracial idempotent states are automatically Haar. In fact, we see that $\widehat{\H}$ is a Kac type closed quantum subgroup of $\widehat{\G}$ if and only if $\omega_{\widehat{\G}/\widehat{\H}} = h_{\widehat{\H}}\circ\pi^u_{\widehat{\H}}$ is tracial. Then if $\G$ is unimodular, we find that an idempotent state is Haar if and only if it is tracial. Furthermore, if $\widehat{\G}/\widehat{\H}$ is coamenable then $\omega_{\widehat{\G}/\widehat{\H}} = h_{\widehat{\H}}\circ\pi^r_{\widehat{\H}}$, where $\omega_{\widehat{\G}/\widehat{\H}}$ and $h_{\widehat{\H}}$ are taken as states on $C_r(\widehat{\G})$ and $C_r(\widehat{\H})$ respectively and $\pi^r_{\widehat{\H}}$ is the reduced version of $\pi^u_{\widehat{\H}}$ due to Theorem \ref{quotientcoamenable} (see Remark \ref{Remark on Idempotent States}).
\begin{rem}\label{TraceImpliesIdemp}
    If we have a tracial state $\tau\in C_u(\widehat{\G})^*$, then the idempotent state obtained by taking a weak$^*$ cluster point of the Cesaro sums $\frac{1}{n}\sum^n_{k=1}\tau^{*k}$ is a tracial idempotent state (see \cite[Lemma 2.1]{NSSS19}). In particular, if a tracial state exists, then a tracial Haar idempotent exists.
\end{rem}
We remind the reader that in what follows we let $(\G_{Kac})_F := \widehat{(\widehat{\G}_{Kac})_F}$ denote the cokernel of the Furstenberg boundary of $\G_{Kac}$.
\begin{prop}\label{tracepropertiesDQGs}
    The following hold:
    \begin{enumerate}
        \item $C_r(\widehat{\G})$ has a tracial state if and only if there exists a Kac type closed quantum subgroup $\widehat{\H}$ of $\widehat{\G}$ such that $\widehat{\G}/\widehat{\H}$ is co-amenable;
        
        \item $C_r(\widehat{\G})$ has a unique tracial state if and only if $\widehat{\G}_{Kac}$ is the only Kac type closed quantum subgroup $\widehat{\H}$ of $\widehat{\G}$ such that $\widehat{\G}/\widehat{\H}$ is coamenable and $C_{rKac}(\widehat{\G}_{Kac}) = C_r(\widehat{\G}_{Kac})$;
        
        \item $C_r(\widehat{\G})$ has a unique idempotent tracial state if and only if $\widehat{\G}_{Kac}$ is the only Kac type closed quantum subgroup $\widehat{\H}$ of $\widehat{\G}$ such that $\widehat{\G}/\widehat{\H}$ is coamenable.
    \end{enumerate}
\end{prop}
\begin{proof}
    1. We are using Theorem \ref{quotientcoamenable} which states that $\widehat{\G}/\widehat{\H}$ is coamenable if and only if $\omega_{\widehat{\G}/\widehat{\H}}\in C_r(\widehat{\G})^*$. As discussed in Remark \ref{TraceImpliesIdemp}, a tracial Haar idempotent exists, and hence so does a Kac type closed quantum subgroup with a coamenable quotient. Conversely, $\omega_{\widehat{\G}/\widehat{\H}}$ is a tracial state in $C_r(\widehat{\G})^*$. 
    
    2. Suppose $C_r(\widehat{\G})^*$ has a unique tracial state $\tau$. Let $\widehat{\H}$ be as in the statement. Then $h_{\widehat{\G}_{Kac}}\circ \pi^r_{\widehat{\G}_{Kac}} = \omega_{\widehat{\G}/\widehat{\G}_{Kac}} = \tau =\omega_{\widehat{\G}/\widehat{\H}} = h_{\widehat{\H}}\circ\pi^r_{\widehat{\H}}$ (see the discussion above Remark \ref{TraceImpliesIdemp}) since $\omega_{\widehat{\G}/\widehat{\G}_{Kac}}$ and $\omega_{\widehat{\G}/\widehat{\H}}$ tracial and by uniqueness of trace. Then, arguing as in Remark \ref{Comparison of Quotients}, we have that $\omega_{\widehat{\G}/\widehat{\G}_{Kac}} = \omega_{\widehat{\G} /\widehat{\H}} \implies \widehat{\H} = \widehat{\G}_{Kac}$.
    
    Conversely, let $\tau\in C_r(\widehat{\G})^*$ be a tracial state. The closed quantum subgroups $\widehat{\H}$ of $\widehat{\G}_{Kac}$ are the Kac type closed quantum subgroups of $\widehat{\G}$. Moreover, if $\widehat{\G}_{Kac}/\widehat{\H}$ were to be coamenable, then so would be $\widehat{\G}/\widehat{\H}$. Indeed, this follows by observing that the composition of the maps 
    $$C_r(\widehat{\G}) \to C_r(\widehat{\G}_{Kac}) \to C_r(\widehat{\H})$$
    gives a $*$-homomorphism as in the statement of Theorem \ref{quotientcoamenable}. Note that since $\widehat{\G}_{Kac}$ is Kac type, we can apply \cite[Corollary 6.7]{ASK23} (see also Remark \ref{Remark Update}) to deduce that $\widehat{\G}_{Kac}/(\widehat{\G}_{Kac})_F$ is coamenable. Hence, $(\widehat{\G}_{Kac})_F = \widehat{\G}_{Kac}$, i.e., the action of $\G_{Kac}$ on its Furstenburg boundary is faithful, so $C_r(\widehat{\G}_{Kac})$ has a unique tracial state by \cite[Theorem 5.3]{KKSV22}. Since $\tau$ is a tracial state, and $I^r_{Kac}\subseteq\ker(\tau)$, using a standard Cauchy-Schwarz argument, there exists a state $\tilde{\tau}\in C_{rKac}(\widehat{\G}_{Kac})^*$ such that $\tilde{\tau}\circ q = \tau$ where $q : C_r(\widehat{\G})\to C_{rKac}(\widehat{\G}_{Kac})$ is the quotient map discussed in the above remark. Clearly $\tilde{\tau}$ is tracial and since $C_{rKac}(\widehat{\G}_{Kac}) = C_r(\widehat{\G}_{Kac})$, $\tilde{\tau} = h_{\widehat{\G}_{Kac}}$ by uniqueness. Therefore, $\tau = \omega_{\widehat{\G}/\widehat{\G}_{Kac}}$.

    3. The proof follows easily from Theorems \ref{HaarIdempotents} and \ref{quotientcoamenable}: the tracial idempotent states in $C_r(\widehat{\G})^*$ are in bijection with Kac type closed quantum subgroups of $\widehat{\G}$ with a coamenable quotient. Furthermore, as shown above, if $C_r(\widehat{\G})$ admits a tracial state then $\widehat{\G}/\widehat{\G}_{Kac}$ is coamenable.
\end{proof}
\begin{rem}\label{RemarkOnKacQuot}
    Notice that we showed that $\widehat{\G}/\widehat{\G}_{Kac}$ is coamenable if there exists a tracial state in $C_r(\widehat{\G})$. In particular, we have that $C_r(\widehat{\G})$ admits a tracial state if and only if $\widehat{\G}/\widehat{\G}_{Kac}$ is coamenable.
\end{rem}
We can use Remark \ref{RemarkOnKacQuot} to partially resolve an open problem from \cite{NV17, C19} at the level of CQGs which generalizes the situation of a discrete group $G$: we have that nuclearity of $C_r(\widehat{G})$ is equivalent to amenability of $G$. Recall that a $C^*$-algebra $A$ is {\bf nuclear} if for every $C^*$-algebra $B$ we have $A\otimes_{min} B = A\otimes_{max} B$.
\begin{cor}\label{nuclear}
    Let $\G$ be a DQG. We have that $C_r(\widehat{\G})$ is nuclear and has a tracial state if and only if $\G$ is amenable.
\end{cor}
\begin{proof}
    Our proof is inspired by the methods used in the proof of \cite[Theorem 5.11]{FSW22}. Note that the proof of \cite[Theorem 5.11]{FSW22} gave a new proof of this fact for locally compact groups. Assume $C_r(\widehat{\G})$ is nuclear and admits a tracial state. Let $\omega_\tau\in C_r(\widehat{\G})^*$ be the idempotent state obtained from the Cesaro sums of $\tau$ as in Remark \ref{TraceImpliesIdemp}. Then, there exists a Kac type quantum subgroup $\widehat{\H}\leq \widehat{\G}$ such that $\omega_\tau = \omega_{\widehat{\G}/\widehat{\H}}$. Using Theorem \ref{quotientcoamenable}, we have that $\widehat{\G}/\widehat{\H}$ is coamenable, and hence $C_r(\widehat{\H})$ is a quotient of $C_r(\widehat{\G})$. It is well-known that quotients of nuclear $C^*$-algebras are nuclear, and so $C_r(\widehat{\H})$ is nuclear as well. Since $\widehat{\H}$ is of Kac type, it follows from \cite[Corollary 3.4]{NV17} that $\widehat{\H}$ is coamenable. So, both $\widehat{\G}/\widehat{\H}$ and $\widehat{\H}$ are coamenable, and we have from \cite[Theorem 3.12]{KKSV22} that $\widehat{\G}$ is coamenable, which entails amenability of $\G$. For the converse, it was shown in \cite{ET03} that coamenability implies nuclearity. Also, the counit $\epsilon_{\widehat{\G}}\in C_r(\widehat{\G})^*$ is a tracial state.
\end{proof}
In the forthcoming section, we will provide a different proof Corollary \ref{nuclear}.

Curiously, the question of whether or not $C_r(\widehat{\G})$ admits a unique tracial state calls into question the exoticness of $C_{rKac}(\widehat{\G}_{Kac})$ as a CQG $C^*$-algebra of $\widehat{\G}_{Kac}$. Unfortunately, it seems we do not even have an example of a non-unimodular DQG $\G$ where $C_r(\widehat{\G})$ has a unique (idempotent) tracial state (see Question \ref{Question 1}). Regardless, an example of a DQG $\G$ where $C_r(\widehat{\G})$ has a unique idempotent tracial state but not a unique tracial state would be an example where $C_{rKac}(\widehat{\G}_{Kac})$ is exotic.
\begin{ques}\label{Question 1}
    If $C_r(\widehat{\G})$ has a unique idempotent tracial state then does $C_r(\widehat{\G})$ have a unique tracial state?
\end{ques}
As curious as Question \ref{Question 1} might be, we think it is imperative that Question \ref{Question 2} is answered first.

\subsection{\texorpdfstring{$\G$}{G}-Invariant States}
An important feature of unimodular DQGs (which was used in \cite{KKSV22} to prove that faithfulness of $\G\acts \partial_F(\G)$ implies the unique trace property) is that tracial states are $\G$-invariant. This allows one to extend tracial states on $C_r(\widehat{\G})$ to $C(\partial_F(\G))\rtimes_r\G$ and is a prevalent technique in group dynamics. In this section we establish that this property holds for arbitrary DQGs. The coaction we consider is the adjoint action $\G\acts C_r(\widehat{\G})$ which is given by $\Delta^l_{\G}(a)= W_\G^*(1\otimes a)W_\G$. We will also need to consider the universal version of the adjoint action $\G\acts C_u(\widehat{\G})$, which is given by $\Delta^{u,l}_{\G}(a) = \mathbb{W}_\G^*(1\otimes a)\mathbb{W}_\G$. Here, $\mathbb{W}_\G$ is the (half-lifted) universal version of $W_\G$ and we have that $(\id\otimes \Gamma_{\widehat{\G}})\circ\Delta^{u,l}_\G(a) = \Delta^l_\G(\Gamma_{\widehat{\G}}(a))$.
\begin{defn}
    A state $\mu\in C_u(\widehat{\G})^*$ is {\bf $\G$-invariant} if $(\id\otimes\mu)\circ\Delta^{u,l}_\G = \mu(\cdot)1$.
\end{defn}
Note that $\G$-invariance of a state $\mu$ on $C_r(\widehat{\G})$ is equivalent to $\G$-invariance of $\mu$ taken as a state on $C_u(\widehat{\G})$, and the universal version of \cite[Lemma 5.2]{KKSV22} holds.
\begin{rem}
    Instead of $\G$-invariance, Crann \cite{C19} used the terminology {\bf inner invariance}. Crann said $\G$ is {\bf topologically inner amenable} if there exists an inner invariant state in $C_r(\widehat{\G})^*$.
\end{rem}
For a discrete group $G$, the $G$-invariant states on $C_u(\widehat{G})$ are exactly the tracial states on $C_u(\widehat{G})$. Generalizing this, for a unimodular DQG $\G$, a state on $C_u(\widehat{\G})$ is tracial if and only if it is $\G$-invariant \cite[Lemma 5.2]{KKSV22}. For non-unimodular $\G$, the (universal) scaling group $(\tau_t)_{t\in\R}$ is non-trivial and we have that the $\G$-invariant states are exactly the $\tau$-KMS states on $C_u(\widehat{\G})$ \cite[Lemma 5.2]{KKSV22}. More precisely, a state $\mu\in C_u(\widehat{\G})^*$ is $\G$-invariant if and only if $\mu(ab) = \mu(\tau_i(b)a)$ for every $a\in C_u(\widehat{\G})$ and $b\in \Pol(\widehat{\G})$.

Let $\mu\in C_u(\widehat{\G})^*$ be a $\G$-invariant state. Let $\omega_\mu$ be a weak$^*$ cluster point of the net of Cesaro sums $(\frac{1}{n}\sum^n_{k=1}\mu^{*k})_{n\geq 1}$. By \cite[Lemma 2.1]{NSSS19}, $\omega_\mu$ is an idempotent state.
\begin{lem}\label{CesaroSumsGInv}
    Let $\mu\in C_u(\widehat{\G})^*$ be a $\G$-invariant state. Then $\omega_\mu \in C_u(\widehat{\G})^*$ is $\G$-invariant and Haar type.
\end{lem}
\begin{proof}
    Since $\mu$ is $\tau$-KMS \cite[Lemma 5.2]{KKSV22}, similarly to tracial states (see Remark \ref{TraceImpliesIdemp}), we find that $\omega_\mu$ is $\tau$-KMS, and hence is a $\G$-invariant idempotent state. Indeed, it it can be shown that $(\tau_z\otimes\tau_z)\circ\Delta^u_{\widehat{\G}} = \Delta^u_{\widehat{\G}}\circ\tau_z$ for every $z\in \C$. Then, for $a\in C_u(\widehat{\G})$ and $b\in \Pol(\widehat{\G})$,
    \begin{align*}
        \mu*\mu(ab) &= (\mu\otimes\mu)\left(\Delta^u_{\widehat{\G}}(a)\Delta^u_{\widehat{\G}}(b)\right)
        \\
        &= (\mu\otimes\mu)\left([(\tau_i\otimes\tau_i)\Delta^u_{\widehat{\G}}(b)]\Delta^u_{\widehat{\G}}(a)\right)
        \\
        &= (\mu\otimes\mu)(\Delta^u_{\widehat{\G}}(\tau_i(b))\Delta^u_{\widehat{\G}}(a))
        \\
        &= \mu*\mu(\tau_i(b)a).
    \end{align*}
    So, the convolution powers of a $\tau$-KMS state is still $\tau$-KMS, and hence so would be $\omega_\mu$.
    
    A straightforward application of the Cauchy-Schwarz inequality informs us that $\omega_\mu$ is a Haar idempotent: indeed, if $a\in I_{\omega_\mu}$, then
    \begin{align*}
        |\omega_\mu(ab^*)|^2 = |\omega_\mu(\tau_i(b^*)a)|^2\leq \omega_\mu(\tau_i(b^*)\tau_i(b^*)^*))\omega_\mu(a^*a) = 0,
    \end{align*}
    for all $b\in \Pol(\widehat{\G})$ and so $a^*\in I_{\omega_\mu}$ by norm density. Then we apply Theorem \ref{HaarIdempotents}. This last computation that shows the null-space of a KMS state is an ideal is well-known (see the proof of \cite[Proposition 7.9]{MVD98}).
\end{proof}
Fix some state $\mu\in C_u(\widehat{\G})^*$ that is either tracial or $\G$-invariant. In what follows, let $\omega_{\widehat{\G}/\widehat{\H}_\mu} \in C_u(\widehat{\G})^*$ be the Haar idempotent state obtained from $\mu$ (see Remark \ref{TraceImpliesIdemp}, Lemma \ref{CesaroSumsGInv}, and Theorem \ref{HaarIdempotents}).
\begin{lem}\label{Traces and Subgroups}
    If a state $\mu\in C_u(\widehat{\G})^*$ is either a tracial or a $\G$-invariant state then there exists a state $\varphi_{\widehat{\H}_\mu}\in C_u(\widehat{\H}_\mu)^*$ such that $\mu = \varphi_{\widehat{\H}_\mu}\circ\pi_{\widehat{\H}_\mu}^u$.
\end{lem}
Before completing the proof, we remind the reader of important features of closed quantum subgroups of CQGs (see \cite{T07} for the compact case and \cite{Daws} for the locally compact case). If $\widehat{\H} \leq \widehat{\G}$, then there is an injective, normal unital $*$-homomorphism $\gamma_\H : \ell^\infty(\H)\to \ell^\infty(\G)$ satisfying $\Delta_\G\circ\gamma_\H = (\gamma_\H\otimes\gamma_\H)\circ\Delta_{\H}$. The pre-adjoint is a surjective algebra homomorphism $(\gamma_\H)_* : \ell^1(\G)\to \ell^1(\H)$. It turns out that $(\gamma_\H\otimes\id)(\mathbb{W}_{\H}) = (\id\otimes\pi_{\widehat{\H}}^u)(\mathbb{W}_\G)$, where $\mathbb{W}_\G$ is the (half-lifted) universal version of $W_\G$ and we write $\mathbb{W}_{\widehat{\G}} = \Sigma\mathbb{W}_\G^*\Sigma$. Note that $\lambda_\G : \ell^1_F(\G)\to \Pol(\widehat{\G})$ is a bijective homomorphism, where $\ell^1_F(\G) = \oplus_{\pi\in \mathrm{Irr}(\widehat{\G})} (M_{n_\pi})_*$. Then, since $\Pol(\widehat{\G})$ identifies as a $*$-subalgebra of $C_u(\widehat{\G})$, $$\pi_{\widehat{\H}_\mu}(\lambda_\G(f)) = \pi_{\widehat{\H}_\mu}((f\otimes \id)(W_\G)) = (f\circ \gamma_{\H_\mu}\otimes\id)(W_\G), ~f\in\ell^1_F(\G).$$
\begin{proof}
    Since $\omega_{\widehat{\G}/\widehat{\H}_\mu}$ is obtained from Cesaro sums of $\mu$, $\mu*\omega_{\widehat{\G}/\widehat{\H}_\mu} = \omega_{\widehat{\G}/\widehat{\H}_\mu}$. Identify $\ell^\infty(\H_\mu)$ with its image $\gamma_{\H_\mu}(\ell^\infty(\H_\mu))\subseteq\ell^\infty(\G)$. We will prove that the condition $\mu*\omega_{\widehat{\G}/\widehat{\H}_\mu} = \omega_{\widehat{\G}/\widehat{\H}_\mu}$ implies the desired result. A proof was provided in a Mathoverflow post of Vaes and later appeared in the proof of \cite[Proposition 6.6]{Mc23}. We will provide a slightly modified version of the proof here. Set $T =\lambda_{\widehat{\G}}(\mu) = (\mu\otimes \id)(\mathbb{W}_{\widehat{\G}})$ and $P_{\H_\mu} = \lambda_{\widehat{\G}}(\omega_{\widehat{\G} / \widehat{\H}_\mu})$. Since $\lambda_{\widehat{\G}}$ is a homomorphism, $TP_{\H_\mu} = P_{\H_\mu}$. Write $R = \mathbb{W}_{\widehat{\G}}(1\otimes P_{\H_\mu}) - 1\otimes P_{\H_\mu}$. Then,
    $$R^*R = 2(1\otimes P_{\H_\mu}) - (1\otimes P_{\H_\mu})\mathbb{W}_{\widehat{\G}}(1\otimes P_{\H_\mu}) -(1\otimes P_{\H_\mu})\mathbb{W}_{\widehat{\G}}^*(1\otimes P_{\H_\mu})$$
    and hence $(\mu\otimes \id)(R^*R) = 0$. Since the null-space of a KMS state (and in particular, a tracial state) is an ideal (in particular, self-adjoint (see the proof of Lemma \ref{CesaroSumsGInv})), $(\mu\otimes \tr_\pi)(RR^*) = 0$ for all $\pi\in \mathrm{Irr}(\widehat{\G})$, where $\tr_\pi \in (M_{n_\pi})_*\subseteq \ell^1(\G)$ is the normalized trace. Hence $(\mu\otimes\id)(RR^*) = 0$ because the $\tr_\pi$'s are faithful. By the Cauchy-Schwarz inequality,
    $$(\mu\otimes\id\otimes\id)(R_{13}(\mathbb{W}_{\widehat{\G}})_{12}) = 0.$$
    Recall that $(\Delta_{\G}\otimes\id)(\mathbb{W}_\G) = (\mathbb{W}_\G)_{13}(\mathbb{W}_\G)_{23}$ which implies $(\id\otimes\Delta_\G)(\mathbb{W}_{\widehat{\G}}) = (\mathbb{W}_{\widehat{\G}})_{13}(\mathbb{W}_{\widehat{\G}})_{12}$. Consequently,
    $$\Delta_\G(T)(1\otimes P_{\H_\mu}) = (\mu\otimes\id\otimes\id)((W_{\widehat{\G}})_{13}(W_{\widehat{\G}})_{12})(1\otimes P_{\H_\mu}) = T\otimes P_{\H_\mu}.$$
    As discussed in Section $2.2$, $T\in \ell^\infty(\H_\mu)$. Furthermore, given $f\in \ell^1_F(\G)$, $\pi_{\widehat{\H}_\mu}(\lambda_\G(f)) = 0$ if and only if $f\circ\gamma_{\H_\mu} = 0$. Hence $\varphi :\Pol(\widehat{\H}_\mu)\to\C$ such that $\varphi(\pi_{\widehat{\H}_\mu}(\lambda_\G(f))) = \mu(\lambda_\G(f))$, $f\in\ell^1_F(\G)$, is a well-defined state, i.e., a unital a functional such that $\varphi(bb^*) \geq 0$ for all $b\in\Pol(\widehat{\G})$. By universality (\cite[Theorem 3.3]{BML01}), $\varphi$ extends continuously to a state $\varphi_{\widehat{\H}_\mu}\in C_u(\widehat{\H}_\mu)^*$ which satisfies $\varphi_{\widehat{\H}_\mu}\circ\pi^u_{\widehat{\H}_\mu} = \mu$ as desired.
\end{proof}
\begin{thm}\label{traceiffGinv}
    A state $\mu\in C_u(\widehat{\G})^*$ is tracial if and only if it is $\G$-invariant.
\end{thm}
\begin{proof}
    Throughout the proof, we maintain the notation established in the paragraph above Lemma \ref{Traces and Subgroups}.
    
    We will begin with a preliminary observation. First, assume that $\mu = \varphi_{\widehat{\H}}\circ\pi_{\widehat{\H}}^u$ for some state $\varphi_{\widehat{\H}}\in C_u(\widehat{\H}_\mu)^*$ and closed quantum subgroup $\widehat{\H}\leq\widehat{\G}$ (we are making no additional assumptions on $\mu$ at the moment). We will prove that $\mu$ is $\G$-invariant if and only if $\varphi_{\widehat{\H}}$ is $\H$-invariant. Accordingly, given $a\in C_u(\widehat{\G})$,
    \begin{align*}
        (\gamma_\H\otimes \varphi_{\widehat{\H}})(\mathbb{W}_{\H}^*(1\otimes \pi_{\widehat{\H} }^u(a))\mathbb{W}_{\H }) &= (\id \otimes \varphi_{\widehat{\H} })((\gamma_{\H }\otimes\id)(\mathbb{W}^*_{\H })(1\otimes \pi_{\widehat{\H} }^u(a))(\gamma_{\H }\otimes\id)(\mathbb{W}_{\H }))
        \\
        &= (\id \otimes \varphi_{\widehat{\H} })((\id\otimes\pi_{\widehat{\H} }^u)(\mathbb{W}^*_\G)(1\otimes \pi_{\widehat{\H} }^u(a))(\id\otimes\pi_{\widehat{\H} }^u)(\mathbb{W}_\G))
        \\
        &= (\id \otimes \varphi_{\widehat{\H} }\circ \pi_{\widehat{\H} }^u)(\mathbb{W}^*_\G(1\otimes a)\mathbb{W}_\G).
    \end{align*}
    This shows,
    \begin{align}\label{EqGHInv}
        (\gamma_{\H }\otimes\varphi_{\widehat{\H}})\circ\Delta_{\H }^{u,l} = (\id\otimes \varphi_{\widehat{\H} }\circ\pi^u_{\widehat{\H} })\circ\Delta^{u,l}_\G.
    \end{align}
    Since $\gamma_{\H }$ is injective, we deduce that $\mu$ is $\G$-invariant if and only if $\varphi_{\widehat{\H} }$ is $\H $-invariant.
    
    Suppose $\mu$ is tracial. Then the corresponding Haar idempotent $\omega_{\widehat{\G}/\widehat{\H}_\mu} = h_{\widehat{\H}_\mu}\circ\pi_{\widehat{\H}_\mu}^u$ is tracial, hence $\widehat{\H}_\mu$ is Kac type. By Lemma \ref{Traces and Subgroups}, $\mu = \varphi_{\widehat{\H}_\mu}\circ \pi^u_{\widehat{\H}_\mu}$ for some state $\varphi_{\widehat{\H}_\mu}\in C_u(\widehat{\H}_\mu)^*$. Clearly, $\varphi_{\widehat{\H}_\mu}$ is tracial hence is $\H_\mu$-invariant by \cite[Lemma 5.2]{KKSV22}. By \eqref{EqGHInv}, $\mu$ is $\G$-invariant. Conversely, suppose $\mu$ is $\G$-invariant. By Lemma \ref{CesaroSumsGInv}, the corresponding Haar idempotnet $\omega_{\widehat{\G}/\widehat{\H}_\mu} = h_{\widehat{\H}_\mu}\circ\pi^u_{\widehat{\H}_\mu}$ is $\G$-invariant. By Lemma \ref{Traces and Subgroups}, $\mu = \varphi_{\widehat{\H}_\mu}\circ \pi^u_{\widehat{\H}_\mu}$ for some state $\varphi_{\widehat{\H}_\mu}\in C_u(\widehat{\H}_\mu)^*$. By \eqref{EqGHInv}, $h_{\widehat{\H}_\mu}$ is $\H_\mu$-invariant, hence $\widehat{\H}_\mu$ is Kac type by \cite[Lemma 5.2]{KKSV22}. Using \eqref{EqGHInv} again, we deduce that $\varphi_{\widehat{\H}_\mu}$ is $\H_\mu$-invariant and hence is tracial by \cite[Lemma 5.2]{KKSV22}. Then, clearly, $\mu$ is tracial too.
\end{proof}
Theorem \ref{traceiffGinv} immediately gives us another proof of Corollary \ref{nuclear}.\\
~\\
{\bf Corollary \ref{nuclear}} {\it We have that $C_r(\widehat{\G})$ is nuclear and has a tracial state if and only if $\G$ is amenable.}
\begin{proof}
    It is was proven in \cite{NV17} (combined with \cite{Toma}) that $C_r(\widehat{\G})$ is nuclear and has a $\G$-invariant state if and only if $\G$ is amenable. The proof is complete with Theorem \ref{traceiffGinv}.
\end{proof}
\begin{rem}
    It is worth noting that there seems to be no example of a DQG $\G$ where $C_r(\widehat{\G})$ is nuclear but $\G$ is non-amenable. Nevertheless, we think our development is interesting because it gives a purely operator algebraic characterization of amenability of $\G$ (in terms of nuclearity). To elaborate, apriori, $\G$-invariance is a quantum group dynamical property (which depends explicitly on the quantum group structure) whereas traciality is an operator algebraic property (which depends only on the ambient operator algebra structure).
\end{rem}

\subsection{Existence and Uniqueness of Traces}
\begin{defn}
    Let $\G$ be a DQG and $\widehat{\H}$ a closed quantum subgroup of $\widehat{\G}$. We say a state $\mu\in C_u(\widehat{\G})^*$ {\bf concentrates} on $\widehat{\G}/\widehat{\H}$ if $\omega_{\widehat{\G}/\widehat{\H}}*\mu = \mu$. 
\end{defn}
\begin{rem}
    The above notion is a noncommutative version of the notion which appears, for example, in the statement of \cite[Theorem 4.1]{BKKO17}. The result of \cite[Theorem 4.1]{BKKO17} is that every tracial state of $C_r(\widehat{G})$, where $G$ is a discrete group, concentrates on the amenable radical of $G$, i.e., $\tau(\lambda(s)) = 0$ for every $s\in G\setminus R_a(G)$. Indeed, using the (contractive) algebraic inclusion $C_r(\widehat{G})^*\subseteq \ell^\infty(G)$,
    \begin{align*}
        \tau(s) = 0 ~ \text{for every $s\notin R_a(G)$} ~\iff \tau(s) 1_{R_a(G)}(s) = \tau(s) ~ \text{for every $s\in G$},
    \end{align*}
    where we recall that in this case $1_{R_a(G)} = \omega_{\widehat{\G}/\widehat{\G}_F}$.
\end{rem}
Recall the definition of $\G_F$ from Section $2.5$, which is the cokernel of the Furstenberg boundary of $\G$. The following was essentially shown in the proof of \cite[Theorem 5.3]{KKSV22}. We complete the proof here. Recall also the discussion in Section $2.5$.
\begin{thm}\label{concentratesonrad}
    Every tracial state $\mu\in C_r(\widehat{\G})^*$ concentrates on $\widehat{\G}/\widehat{\G}_F$.
\end{thm}
\begin{proof}
    Since $\mu$ is tracial, it is $\G$-invariant by Theorem \ref{traceiffGinv}. Using $\G$-injectivity of $C(\partial_F(\G))$, we obtain a ucp $\G$-equivariant extension $M_\mu : C(\partial_F(\G))\rtimes_r \G\to C(\partial_F(\G))$ of $\mu$, where $\mu : C_r(\widehat{\G})\to C(\partial_F(\G))$ is interpretted as the $\G$-equivariant ucp map $a\mapsto \mu(a)1$. It was shown in the proof of \cite[Theorem 5.3]{KKSV22} that for all $y\in \ell^\infty(\G_F)$, $\lambda_{\widehat{\G}}(\mu)y = \epsilon_\G(y)\lambda_{\widehat{\G}}(\mu)$. We remind the reader how this is done here.
    
    By $\G$-rigidity, the restriction of $M_\mu$ to $\alpha(C(\partial_F(\G))$ is equal to $\alpha^{-1}$, so we conclude that $\alpha(C(\partial_F(\G))$ lies in the multiplicative domain of $M_\mu$. In particular,
    $$M_\mu(\alpha(x)(a\otimes 1)) = \mu( a)x = M_\mu(( a\otimes 1)\alpha(x))$$
    for all $x\in C(\partial_F(\G))$ and $ a\in C_r(\widehat{\G})$. Let $\beta$ be the coaction of $\G$ on $C(\partial_F(\G))\rtimes_r\G$ (see Section $2.5$). By definition,
    $$(W_\G^*\otimes 1)(1\otimes \alpha(x)) = \beta(\alpha(x))(W_\G^*\otimes 1), ~ x\in C(\partial_F(\G)).$$
    By applying $\id\otimes M_\mu$ to both sides of the above equation and using $\beta|_{\alpha(C(\partial_F(\G)))} = \id\otimes\alpha$, we obtain,
    $$(\id\otimes \mu)(W_\G^*)\otimes x = \alpha(x)((\id\otimes\mu)(W_\G^*)\otimes 1).$$
    Therefore $\varphi(x)(\id\otimes\mu)(W_\G^*) = \mathcal{P}_\varphi(x)(\id\otimes\mu)(W_\G^*)$ for every $\varphi\in C(\partial_F(\G))^*$ and $x\in C(\partial_F(\G))$. Then, since $\varphi(x) = \epsilon_\G\circ\mathcal{P}_\varphi(x)$ and $W_{\widehat{\G}} = \Sigma(W_\G^*)\Sigma$, we deduce that
    $$\epsilon_\G(y)\lambda_{\widehat{\G}}(\mu) =  \lambda_{\widehat{\G}}(\mu)y, ~ \text{for all} ~ y\in \ell^\infty(\G_F).$$
    In particular, for $P_{\G_F} = \lambda^u_{\widehat{\G}}(\omega_{\widehat{\G}/\widehat{\G}_F})\in \ell^\infty(\G_F)$, $$\lambda^u_{\widehat{\G}}(\omega_{\widehat{\G}/\widehat{\G}_F}*\mu) = P_{\G_F}\lambda_{\widehat{\G}}(\mu) = \lambda_{\widehat{\G}}(\mu).$$
    So, $\omega_{\widehat{\G}/\widehat{\G}_F}*\mu = \mu$.
\end{proof}
\begin{rem}\label{concentratedonradicalremark}
    Let $\H$ be a Kac type closed quantum subgroup of $\widehat{\G}$ such that $\widehat{\G}/\widehat{\H}$ is coamenable. From Theorem \ref{concentratesonrad}, we have that the associated (tracial) Haar idempotent $\omega_{\widehat{\G}/\widehat{\H}}\in C_r(\widehat{\G})^*$ satisfies $\omega_{\widehat{\G}/\widehat{\G}_F}*\omega_{\widehat{\G}/\widehat{\H}} = \omega_{\widehat{\G}/\widehat{\H}}$. By the calculations in Remark \ref{Comparison of Quotients}, $L^\infty(\widehat{\G}/\widehat{\H})\subseteq L^\infty(\widehat{\G}/\widehat{\G}_F)$ and so $\widehat{\G}_F$ is a closed quantum subgroup of $\widehat{\H}$ by Remark \ref{Comparison of Quotients}.
\end{rem}
From here, we obtain our highlighted results on the existence and uniqueness of traces.
\begin{cor}\label{existenceanduniqueoftraces}
    Suppose that $\widehat{\G}/\widehat{\G}_F$ is coamenable. The following hold:
    \begin{enumerate}
        \item $\widehat{\G}_F$ is Kac type if and only if $C_r(\widehat{\G})$ has a tracial state;
        
        \item  $C_r(\widehat{\G})$ has a tracial state, $\widehat{\G}_F = \widehat{\G}_{Kac}$, and the canonical quotient map $C_{rKac}(\widehat{\G}_{Kac}) \to C_r(\widehat{\G}_{Kac})$ is injective if and only if $C_r(\widehat{\G})$ has a unique tracial state;
        
        \item $C_r(\widehat{\G})$ has a tracial state and $\widehat{\G}_F = \widehat{\G}_{Kac}$ if and only if $C_r(\widehat{\G})$ has a unique idempotent tracial state.
    \end{enumerate}
\end{cor}
\begin{proof}
    1. If $C_r(\widehat{\G})$ has a tracial state, then Proposition \ref{tracepropertiesDQGs} implies that a Kac type closed quantum subgroup $\widehat{\H}$ such that $\widehat{\G}/\widehat{\H}$ is coamenable exists. Then $\widehat{\G}_F$ is Kac type because it is a closed quantum subgroup of $\widehat{\H}$. Conversely, if $C_r(\widehat{\G})$ has no tracial states, then $\widehat{\G}_F$ could not be Kac type because then $\omega_{\widehat{\G}/\widehat{\G}_F}\in C_r(\widehat{\G})^*$ would be tracial.
    
    2. Thanks to Proposition \ref{Coamenable Kac Quotient}, $\widehat{\G}/\widehat{\G}_{Kac}$ is coamenable. Then, as discussed in Remark \ref{concentratedonradicalremark}, $\widehat{\G}_F$ is a closed quantum subgroup of $\widehat{\G}_{Kac}$. If $C_r(\widehat{\G})$ has a unique tracial state, then $h_{\widehat{\G}_F}\circ\pi^r_{\widehat{\G}_F} = \omega_{\widehat{\G}/\widehat{\G}_F} = \omega_{\widehat{\G}/\widehat{\G}_{Kac}} = h_{\widehat{\G}_{Kac}}\circ\pi^r_{\widehat{\G}_{Kac}}$ (see the discussion above Remark \ref{TraceImpliesIdemp}), so $\widehat{\G}_F = \widehat{\G}_{Kac}$. Moreover, since $\omega_{\widehat{\G}/\widehat{\G}_F}\in C_r(\widehat{\G})^*$ is the unique tracial state,
    $$I^r_{Kac} = \{a\in C_r(\widehat{\G}) : \omega_{\widehat{\G}/\widehat{\G}_F}(a^*a) = 0\} =: I^r_{\omega_{\widehat{\G}/\widehat{\G}_F}}$$
    and $C_r(\widehat{\G})/I^r_{\omega_{\widehat{\G}/\widehat{\G}_F}} = C_r(\widehat{\G}_F)$ (see the discussion following Theorem \ref{HaarIdempotents}).
    
    Conversely, for every Kac type closed quantum subgroup $\widehat{\H}$ of $\widehat{\G}$ where $\widehat{\G}/\widehat{\H}$ is coamenable, $L^\infty(\widehat{\G}/\widehat{\G}_{Kac})\subseteq L^\infty(\widehat{\G}/\widehat{\H})\subseteq L^\infty(\widehat{\G}/\widehat{\G}_F)$ (Remark \ref{Comparison of Quotients}). So, $\widehat{\H} = \widehat{\G}_F = \widehat{\G}_{Kac}$. Then Theorem \ref{tracepropertiesDQGs} implies $C_r(\widehat{\G})$ has a unique tracial state.
    
    3. Similar to the proof of 2..
\end{proof}
\begin{rem}\label{Remark Update}
    It follows from \cite[Corollary 6.7]{ASK23} that $\widehat{\G}/\widehat{\G}_F$ is coamenable in the following situations: whenever $\widehat{\G}$ is Kac type and whenever $C_r(\widehat{\G})$ is exact (as a $C^*$-algebra). In particular, with reference to Corollary \ref{existenceanduniqueoftraces}, the assumption of coamenability of $\widehat{\G}/\widehat{\G}_F$ is automatic in the case where $C_r(\widehat{\G})$ is exact.
\end{rem}
Our last result calls into question how often $C_r(\widehat{\G})$ admits a unique (idempotent) tracial state. At the moment, it seems we have no examples at all. In principle, it might be the case that if $C_r(\widehat{\G})$ has a unique (idempotent) trace then $\G$ is unimodular, but such a result seems rather strong. Nevertheless, we cannot yet rule out the truth of such a statement.
\begin{ques}\label{Question 2}
    Is there an example of a non-unimodular DQG $\G$ for which $C_r(\widehat{\G})$ has a unique (idempotent) tracial state?
\end{ques}

{\flushleft\bf Acknowledgements}: We thank Fatemeh Khosravi for helpful conversations on nuclear $C^*$-algebras. We also thank Adam Skalski for some clarification regarding the canonical Kac quotient of a compact quantum group. We thank the referee for their suggestions. They greatly improved the presentation of this paper. The bulk of this article was completed as part of the author's doctoral work. For part of this work, the author was supported by the ANR project ANR-19-CE40-0002.

{\small\bibliography{TracesAndCoamenableCoideals}}
\bibliographystyle{abbrv}
\renewcommand*{\bibname}{References}

\end{document}